\documentclass[a4paper,11pt]{amsart}
\usepackage{amsmath,amsfonts,amscd,amssymb,amsthm}
\usepackage{a4wide}
\usepackage{enumerate}
\usepackage{longtable}
\usepackage[english]{babel}
\usepackage[utf8x]{inputenc}
\usepackage[active]{srcltx}
\usepackage[T1]{fontenc}
\usepackage{bbm}
\usepackage{leftidx}
\usepackage{MnSymbol}
\usepackage{stmaryrd}
\usepackage{subcaption}
\usepackage{latexsym}
\usepackage{mathtools}
\usepackage{caption}

\theoremstyle{plain}
\newtheorem*{theorem-non}{Theorem}
\newtheorem*{thm}{Main Theorem}
\newtheorem*{lem}{Lemma \ref{2fusion}}
\newtheorem{theorem}{Theorem}[section]
\newtheorem{conj}{Conjecture}
\newtheorem{corollary}[theorem]{Corollary}
\newtheorem{lemma}[theorem]{Lemma}
\newtheorem{proposition}[theorem]{Proposition}

\newtheorem{remark}[theorem]{Remark}

\theoremstyle{definition}

\newtheorem{definition}{Definition}[section]
\newtheorem{example}[theorem]{Example}

\newcommand{\be}[1]{\begin{equation}\label{#1}}
\newcommand{\ee}{\end{equation}}
\numberwithin{equation}{section}
\renewenvironment{proof}[1][\relax]
  {\paragraph{\textbf{Proof}\ifx#1\relax\else~of #1\fi}}%
  {~\hfill$\square$\par\bigskip}
   
\newcommand{\calF}{\mathcal{F}}
\newcommand{\calG}{\mathcal{G}}

\newcommand{\bbZ}{\mathbb{Z}}

\renewcommand{\be}{\beta}

\newcommand{\khpqd}[2]{\widetilde{Kh}^{\ast}_{\ast}(T_{#1,#2})}
\newcommand{\khpq}[2]{\widetilde{Kh}^{\ast,\ast}(T_{#1,#2})}
\newcommand{\khr}{\widetilde{Kh}}
\renewcommand{\Im}{\textrm{Im}}

\title{An algebra structure for the stable Khovanov homology of torus links}

\begin{document}

\author{Mounir Benheddi}
\address{Universit\'e de Gen\`eve, Section de math\'ematiques, 2-4 rue du Li\`evre, 1211 Gen\`eve 4, Switzerland}
\email{Mounir.Benheddi@unige.ch}

\begin{abstract}
The family of negative torus links $T_{p,q}$ over a fixed number of strands $p$ admits a stable limit in reduced Khovanov homology as $q$ grows to infinity. In this paper, we endow this stable space with a bi-graded commutative algebra structure. We describe these algebras explicitly for $p=2,3,4$. As an application, we compute the homology of two families of links, and produce a lower bound for the width of the homology of any $4$-stranded torus link.
\end{abstract}

\maketitle

\section*{Introduction} 

Introduced by Khovanov \cite{khovanov2000}, Khovanov homology is a bi-graded homology theory that generalizes the Jones polynomial. More precisely, its graded Euler characteristic gives back the Jones polynomial. It is thus natural to ask oneself the following question:
\medskip

\emph{If there is an explicit formula for the Jones polynomial of a given family of links, is there also an explicit description of the Khovanov homology of that family?}

\medskip
One such example, which will be our focus in this paper, is the family of negative torus links $T_{p,q}$. For $p=2, 3$, their homologies $\khr^{\ast,\ast}(T_{p,q})$ have been computed for all $q$ respectively by Khovanov \cite{khovanov2000} and Turner \cite{turner2008spectral} or Stosic \cite{stovsic2009khovanov}. However such computations are notoriously hard: even for computers, there is a threshold beyond which computations are not possible.
One approach to study these families is to look at another theory, the triply graded homology, which generalizes the HOMFLY-PT polynomial. In this realm, our earlier question has been investigated with much success by Hogancamp et al. (\cite{elias2016computation},\cite{hogancamp2017khovanov} amongst others) or more recently by Mellit \cite{mellit2017homology}. 

Another approach is to regularize the homologies of torus links over a fixed number of strands by increasing the number of twists. The question of the existence for this \emph{stable} Khovanov homology of torus links $\khr^{\ast,\ast}(T_{p,\infty})$ was asked by Dunfield, Gukov and Rasmussen \cite{dunfield2006superpolynomial} and answered positively by Sto\v{s}i\'c \cite{stovsic2007homological}. These stable spaces exhibit very nice properties and are known to generalize the Jones-Wenzl projectors (Rozansky \cite{rozansky2010infinite}). As surprising as it may sound, they are actually easier to compute than their finite counterparts. Moreover there is a beautiful conjecture, due to Gorsky, Oblomkov and Rasmussen \cite{gorsky2013stable}, proved by Hogancamp \cite{hogancamp2015stable} for the triply graded homology. It not only predicts what these spaces should be, but also that they should carry an \emph{extra structure}, that of an algebra. We will follow that approach and restrict our attention to the reduced Khovanov homology with $\bbZ_2$ coefficients. The main feature we use in our work is the functoriality of Khovanov homology. For each $p\geq 2$, we endow $\khr^{\ast,\ast}(T_{p,\infty})$ with an algebra structure that originates from well-chosen sequences of diagrams, or \emph{movies} for short. More precisely we have:

\begin{thm}
	There exists a bi-degree $(0,0)$ map (to be defined)
	\[
	 \widetilde{\Sigma}^p_{\infty}: \khr^{\ast,\ast}(T_{p,\infty}) \otimes \khr^{\ast,\ast}(T_{p,\infty}) \longrightarrow \khr^{\ast,\ast}(T_{p,\infty})
	\]
	that endows $\khr^{\ast,\ast}(T_{p,\infty})$ with a bi-graded commutative unitary algebra structure.
\end{thm}

We will be mainly interested in the $2,3$ and $4$-stranded cases. We have been able to determine the stable algebra for $p=2$, $p=3$. The $p=4$ case is also treated, up to a small indeterminacy. We will show the following, where $\widetilde{Kh}^{\ast}_{\ast}(D)$ denotes the $\delta$-graded version of Khovanov homology:

\begin{theorem-non}[Theorems \ref{Th2inf}, \ref{Algebra3}, \ref{Alg4}]
	There are bi-graded algebras isomorphisms:
	\begin{itemize}
		\item {$\widetilde{Kh}^{\ast}_{\ast}(T_{2,\infty}) \cong \mathbb{Z}_2[x,y]/(x^3=y^2)$.}
		\item {$\widetilde{Kh}^{\ast}_{\ast}(T_{3,\infty}) \cong \mathbb{Z}_2[x,y,z]/(x^2=y^2=0)$.}
		\item {$\widetilde{Kh}^{\ast}_{\ast}(T_{4,\infty}) \cong \dfrac{\mathbb{Z}_2[x,y,z,v,w]}{(x^2=y^2=0, w^2 = \alpha vz^2 + \beta xv^2 )}$,}
		where $\alpha,\beta \in \{0,1\}$.
	\end{itemize}	
	The degrees of the generators are given by
	\[
	|x|=(-2,0), \hspace{0.1cm} |y|=(-3,0), \hspace{0.1cm} |z|=(-4,2), \hspace{0.1cm} |v|=(-6,4), \hspace{0.1cm} |w|=(-7,4).
	\]
\end{theorem-non}

If the first two cases rely heavily on explicit knowledge of $\khr^{\ast,\ast}(T_{2,q})$ and $\khr^{\ast,\ast}(T_{3,q})$ for all $q \geq 1$, we were able to work out the $p=4$ case using only descriptions of $\khr^{\ast,\ast}(T_{4,4})$ and $\khr^{\ast,\ast}(T_{4,5})$. Along the way, we use these structures to compute various homologies and estimate the width of any $4$-stranded torus links. Finally, the reader familiar with the G-O-R conjecture will already notice a difference between our result and the predicted outcome for the $2$ stranded case. This discrepancy shall be addressed.

\subsection*{Outline of the paper}
In Section 1, we give a definition of the chain complex for \emph{reduced} Khovanov homology and explore the naturality of some well-know properties such as the independence from a choice of basepoint, the monoidality with respect to connected sums and the long exact sequence. Once these tools are in place, we present the construction of the stable spaces $\khr^{\ast,\ast}(T_{p,\infty})$. 

In Section 2, we introduce an algebra structure on $\khr^{\ast,\ast}(T_{p,\infty})$ for fixed $p$, as a limit of maps induced by a succession $1$-handle moves on diagrams and compute it explicitly for $p=2$. In a second time, we use tangles to define modules over these algebras.

In Section $3$, we use modules over $\khr^{\ast,\ast}(T_{2,\infty})$ as a shortcut to compute the homology of two families of braid closures as well as $\khr^{\ast,\ast}(T_{3,\infty})$ as an algebra. 

In Section $4$, we focus on the $4$-stranded case, and derive a lower bound on the width of $\khr^{\ast,\ast}(T_{4,q})$ for all $q$ from the algebra structure. Finally we discuss the relationship between our result and the Gorsky-Oblomkov-Rasmussen conjecture. 

In Section $5$, we wrap up loose ends, i.e. we introduce a technique to understand maps induced by $1$-handle moves. That technique will prove very useful for discussing surjectivity of the maps that yield the algebra structures.

\subsection*{Acknowledgements}
The author is grateful to his PhD advisors David Cimasoni and Paul Turner, for their teachings and advices, as well as precious comments on an earlier version. This work is supported by the Swiss FNS.

\section{Reduced Khovanov homology}

In this introductory section, we briefly describe a chain complex for reduced Khovanov homology over $\bbZ_2$ and present selected properties. In a second part, we explore the naturality of these properties. Finally, we define the stable limit $\khpq{p}{\infty}$.

\subsection{Background}
To fix notations, we dedicate this subsection to a construction of the (reduced) Khovanov chain complex over $\bbZ_2$. We follow Turner's Hitchhiker's guide \cite{turner2014hitchhiker}.

Let $D$ be an oriented link diagram. Denote by $\chi_D$ the set of 
crossings of $D$. Any crossing $c \in \chi_D$ can be smoothed in two different ways, described below.

\medskip

\begin{center}
	\includegraphics{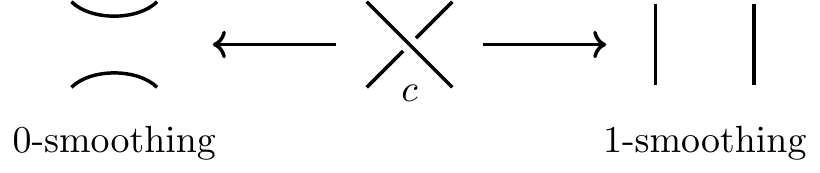}
\end{center}

\begin{definition}
	A \emph{smoothing} $s$ of a diagram $D$ is a map $s: \chi_D \rightarrow \{0,1\}$. By extension, the crossing-less diagram (i.e a collection of circles) obtained from $D$ by replacing each crossing $c$ with the $s(c)$-smoothing will also be called \emph{smoothing}. For a given smoothing $s$, let $| s |$ denote the number of circles in that crossing-less diagram.
\end{definition}

To any subset $A \subset \chi_D$, one can associate a smoothing $s_A$ defined by:
\[
\begin{array}{cccc}
s_A: &  \chi_D & \rightarrow & \{0,1\} \\
& c & \mapsto & 
\begin{dcases*}

0 & if $c \notin A$ \\
1 & if $c \in A$. \\
\end{dcases*}
\end{array}
\]

Thus to any such subset, one can associate a vector space 
\[
V_A := \bbZ_2 \{ x_{\gamma} \hspace{0.1cm}|\hspace{0.1cm} \gamma \mbox{ circle in } s_A\}
\]
and the exterior algebra $\Lambda V_A$ over $V_A$. This algebra is generated by symbols $x_{\gamma_1} \wedge \ldots \wedge x_{\gamma_k}$, such that
$x \wedge x=0$ for any $x \in \Lambda V_A$. Since we work with $\bbZ_2$ coefficients, the equality $x \wedge y = y \wedge x$ holds. Hence we shall supress the wedge notation and adopt a polynomial notation $xy := x \wedge y$. In particular, we now have $x_{\gamma}^2=0$ for any circle $\gamma$.

Before we construct the Khovanov chain complex, let us define one more object.

\begin{definition}
	Let $W= \bigoplus\limits_{i_1, i_2} W^{i_1, i_2}$ be a $\bbZ^{2}$-graded vector space. For $(k_1, k_2) \in \bbZ^2$, the \emph{shifted vector space} $W[k_1, k_l]$ is the $\bbZ^2$-graded vector space defined by
	\[
	W[k_1, k_2] = \bigoplus\limits_{i_1, i_2 \in \bbZ} W^{i_1,i_2}[k_1, k_2], \mbox{ where } W^{i_1, i_2}[k_1,  k_2]:= W^{i_1-k_1, i_2-k_2}.
	\]
\end{definition}

Given an oriented link diagram $D$, let $n_{-}(D)$ (resp. $n_{+}(D)$) be the number of negative (resp. positive) crossings of $D$. The algebra $\Lambda V_A$ can be $\bbZ$-graded: consider the \emph{quantum grading}
	\[
	q: \Lambda V_A \longrightarrow \bbZ
	\]
	defined on monomials as
	\[
	q(x_{\gamma_1}\cdots x_{\gamma_k}) = |s_A| - 2k + |A| + n_{+}(D) - 2n_{-}(D).
	\]
	A monomial $v \in \Lambda V_A$ is said to have \emph{quantum degree} $j$ if $q(v)=j$. We will denote by $\Lambda^q V_A \subset \Lambda V_A$ the subspace generated by monomials with quantum degree $q$, so that 
	\[
	\Lambda V_A =  \bigoplus_{q} \Lambda^q V_A.
	\]
The reader should note that this quantum grading is not the usual grading on the exterior algebra.

We can now turn to the main object of this section, the Khovanov chain complex.
For $i \in \bbZ$, define the vector space
\[
C^{i}(D)= \bigoplus_{\substack{A\subset \chi_D \\ |A| = i + n_{-}(D)}} \Lambda V_A.
\]
Any element $v \in C^{i}(D)$ is said to have \emph{homological degree} $i$. Each $C^i(D)$ inherits a quantum grading from the $\Lambda V_A$'s. An element $v \in C^{i,j}(D)$ is said to have bi-degree $|v|=(i,j)$ if it has homological degree $i$ and quantum degree $j$.

\medskip

The family $\{C^i(D)\}_{i=-n_{-}(D), \ldots, n_{+}(D)} $ can be endowed with a structure of chain complex. Suppose we have $A \subset B \subset \chi_D$, such that $|B| = |A| + 1$. Then the two smoothings $s_A$ and $s_B$ are identical except in a small disk centered around the crossing $c \in B \backslash A$.
Diagrammatically, we have:

\begin{center}
	\includegraphics{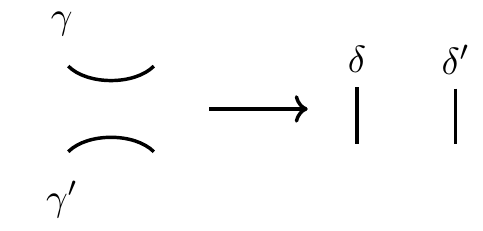}
\end{center}

There are exactly two possible configuration of circles, which allow us to define two maps.
\begin{enumerate}[(i)]
	\item{ If $\gamma \neq \gamma'$, then $\delta = \delta'$. Define the product
		\[
		\begin{array}{cccc}
		m_{A,B}: & \Lambda V_A  & \longrightarrow & \Lambda V_B  \\
		& v & \longmapsto & v, \\
		& x_{\gamma}v,\mbox{ } x_{\gamma'}v & \longmapsto & x_{\delta}v, \\
		& x_{\gamma}x_{\gamma'}v & \longmapsto & 0. \\		         		         
		\end{array}
		\]
	}
	\item{If $\gamma = \gamma'$, then $\delta \neq \delta'$. Define the co-product
		\[
		\begin{array}{cccc}
		\Delta_{A,B}: & \Lambda V_A & \longrightarrow & \Lambda V_B\\
		&  v  & \longmapsto & (x_{\delta} + x_{\delta'})v, \\
		& x_{\gamma}v & \longmapsto & x_{\delta}x_{\delta'}v. \\
		\end{array}            
		\]
	}	
\end{enumerate}

Here we assume $v$ to be an element of $\Lambda V_A$ which contains neither $x_{\gamma}$ nor $x_{\gamma'}$, and extend the maps linearly to $\Lambda V_A$.
Finally we define $d^i:C^i(D) \longrightarrow C^{i+1}(D)$ by extending linearly:
\[
d^i(v) = \sum_{\substack{A\subset B \subset \chi_D \\ |B| =1+ i + n_{-}(D)}} d_{A,B}(v),
\]
where $v \in \Lambda V_A \subset C^i(D)$ and $d_{A,B}$ is either $m_{A,B}$ (if $|s_A|>|s_B|)$ or $\Delta_{A,B}$ (if $|s_A|<|s_B|)$.

\begin{theorem}\cite{khovanov2000}
	The maps $d^i$ have bi-degree $(1,0)$ and for all $i \in \bbZ$ we have $d^{i+1} \circ d^i =0$. Moreover, the isomorphism class of the bi-graded homology $Kh^{\ast, \ast}(D) = H^{\ast}(C^{\ast, \ast}(D), d)$ is an invariant of oriented link diagrams.
\end{theorem}

In fact, Khovanov homology is more than an invariant, it is a functor. Let us give the definition of the maps associated with $1$-handle moves. Such a move does not involve any crossing, so $\chi_{D} = \chi_{D'}$. Thus any $A \subset \chi_{D}$ gives rise to two smoothings: $s_A$ for $D$ and $s_A'$ for $D'$, as well as two corresponding vector spaces $\Lambda V_A \subset C^{i}(D)$ and $\Lambda V_A' \subset C^{i}(D')$. Define 
		\[
		\phi_A:\Lambda V_A \rightarrow \Lambda V_A'
		\]
		by the same formula as $m_{A,B}$ or $\Delta_{A,B}$, depending on whether $|s_A| > |s_A'|$ or $|s_A| < |s_A'|$.	Let 
		\[
		\phi^i =  \bigoplus_{\substack{A\subset \chi_D  \\ |A| = i + n_{-}(D)}} \phi_A.
		\]
		The map $\phi^i$ has bi-degree $(0,-1)$ and the family $\{\phi^i\}$ is a chain map. Finally, set
		\[\phi: Kh^{i,j}(D) \rightarrow Kh^{i,j-1}(D')
		\] to be the map induced on homology.

\medskip

Let $(D,p)$ be a pair of an oriented diagram $D$ and a point $p \in D$, which is not a double point. In this context, any smoothing of $D$ is a collection of circles containing exactly one pointed circle. For every $A \subset \chi_D$, let $x_{\bullet}$ be the variable of $\Lambda V_A$ associated to the pointed circle and, in a slight abuse of notations, $x_{\bullet} : C^{\ast,\ast}(D) \rightarrow C^{\ast,\ast - 2}(D)$ be the chain map defined by $x_{\bullet}(v)= x_{\bullet}v$.

It is well know that $x_{\bullet}C(D)=\Im(x_{\bullet})$ is a sub-complex of $(C(D), d)$. The \emph{reduced Khovanov chain complex}, denoted by $\widetilde{C}^{\ast, \ast}(D,p)$, is defined as a shifted version of this sub-complex:
\[
\widetilde{C}^{\ast, \ast}(D,p) =  x_{\bullet}C^{\ast,\ast}(D)[0,1].
\]
The \emph{reduced Khovanov homology} of the pair $(D,p)$ is the vector space defined by
\[
\khr^{\ast, \ast}(D,p) := H_{\ast}(\widetilde{C}^{\ast, \ast}(D,p)).
\] 
The maps induced by $1$-handle in the reduced theory are just given by restriction to the shifted sub-complex $\widetilde{C}^{\ast, \ast}(D,p)$.

\medskip

There is a clear relationship between the unreduced and reduced Khovanov homologies, made explicit by Shumakovitch \cite{shumakovitch2004torsion}, as a split short exact sequence
\begin{center}
	\includegraphics{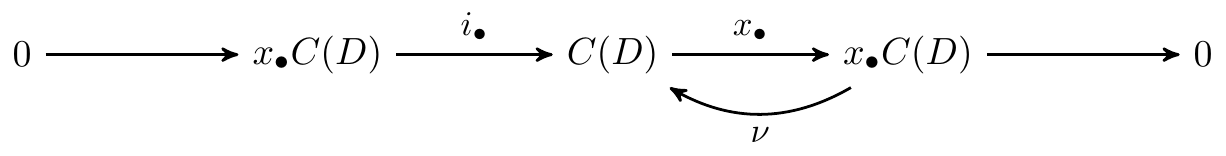}
\end{center}
where $i_{\bullet}$ is the inclusion of subcomplex and $\nu$ is defined below. This exact sequence induces an isomorphism in homology
\[
Kh^{\ast,\ast}(D) \cong \widetilde{Kh}^{\ast,\ast}(D,p)[0,1] \bigoplus \widetilde{Kh}^{\ast,\ast}(D,p)[0,-1].
\] 
This shows independence from the choice of basepoint. Let us recall the definition of Shumakovitch's map $\nu: C^{\ast,\ast}(D) \rightarrow C^{\ast, \ast + 2}(D)$. For $x_{\gamma_1}\cdots x_{\gamma_k} \in \Lambda V_A$, define 
\[
\nu (v) = \sum\limits_{i=1}^{k} x_{\gamma_1}\cdots \hat{x}_{\gamma_i} \cdots x_{\gamma_k},
\] and extend this map linearly to the whole complex. The map $\nu$ is a chain map, and has the property that $\nu \circ \nu = 0$.

\medskip

Let us say a final word about the construction. From a computational point of view, it is often wise to consider another grading, the so-called \emph{$\delta$-grading}. It is defined by the formula
\[
\delta =j-2i.
\]
The associated homology will be denoted by $\khr^{\ast}_{\ast}(D,p)$. An element $v \in \khr^{i}_{\delta}(D,p)$ will have \emph{bi-degree} $(i,\delta)$. In this work we will use the quantum graded version for constructions, and the $\delta$-graded version for computations, so there should be no confusion as to which bi-degree we refer to. This grading allows to introduce a particular kind of diagrams.

\begin{definition}
	Let $V= \bigoplus V^{i}_{\delta}$ be a bi-graded vector space and
	\[
	\delta_{\mbox{ max }} = \mbox{ max } \{ \delta \hspace{0.25cm} | \hspace{0.25cm} V^{\ast}_{\delta} \neq \{0\}\}, \hspace{0.25cm}
	\delta_{\mbox{ min }} = \mbox{ min } \{ \delta \hspace{0.25cm} | \hspace{0.25cm} V^{\ast}_{\delta} \neq \{0\}\}.
	\]
	We define the width of $V$ as the integer
	\[
	w(V) = \frac{1}{2} ( \delta_{\mbox{ max }} - \delta_{\mbox{ min }}) + 1.
	\]
	We say that a diagram $D$ is \emph{thin} if $w(\khr^{\ast}_{\ast}(D))=1$, and \emph{thick} otherwise.
\end{definition}

\subsection{Naturality properties}
Since the functoriality property of Khovanov homology is essential to our work, we need to study various naturality properties. We first how maps induced by $1$-handle moves change with respect to different choices of basepoint.

\begin{lemma}
	Let $D$ be an oriented diagram. For any two choices of basepoints $p$ and $p'$, there is an isomorphism
	\[
	f:\widetilde{C}^{i,j}(D,p) \longrightarrow \widetilde{C}^{i,j}(D,p'),
	\]
	which is natural with respect to maps induced by $1$-handle moves.
\end{lemma}

\begin{proof}
	Let $x_{\bullet}$ (resp. $x_{\bullet}'$) be the variable associated to the circles containing $p$ (resp. $p'$). Consider the composite chain maps
	\[
	\begin{array}{c}
	f := x_{\bullet}' \circ \nu : x_{\bullet}C^{\ast, \ast}(D,p) \longrightarrow C^{\ast,j+2}(D) \longrightarrow x_{\bullet}'C^{\ast, \ast}(D,p'), \\
	g := x_{\bullet} \circ \nu : x_{\bullet}'C^{\ast, \ast}(D,p') \longrightarrow C^{\ast,j+2}(D) \longrightarrow x_{\bullet}C^{\ast, \ast}(D,p).
	\end{array}
	\]
	The proof of the equalities $f \circ g = id$ and $g \circ f = id$ is a direct computation. The naturality is a by-product of the fact that both $x_{\bullet}$ and $\nu$ are chain maps: they both commute with the product $m$ and co-product $\Delta$. Since maps induced by $1$-handle moves are made only of such pieces, it follows that they commute with both chain maps.
\end{proof}
	
Apart from the (natural) independence of basepoint, the reduced Khovanov homology has a very interesting property with respect to connected sums of links, which we explore next. Of course, the notion of connected sum of links is usually not well defined, however, with the additional data of basepoints, there is a canonical choice of which components to connect. The original result, due to Khovanov \cite{khovanov2003patterns}, is stated for the non-reduced version as an isomorphism of $A$-modules, where $A=\bbZ[x_{\bullet}]/(x_{\bullet}^2)$:
\[
C(D \sharp D') \cong C(D) \otimes_{A} C(D'),
\]
In our context of reduced homology, the result is slightly different.

\hspace*{-2in}
\begin{figure}
	\hspace*{-1.25in}
	\begin{subfigure}{.5\textwidth}
		\centering
		\includegraphics[scale=0.75]{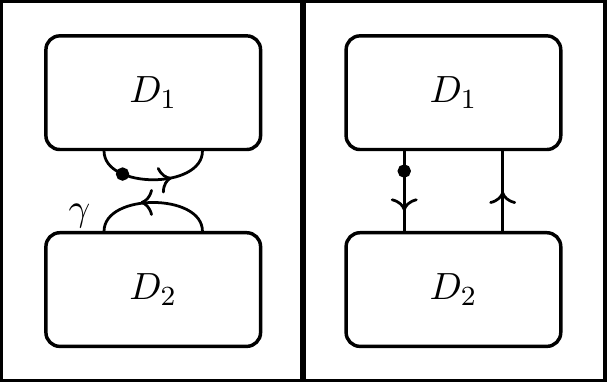}
		\caption{A $1$-handle realizing a connected sum}
		\label{fig:sub1}
	\end{subfigure}
	\begin{subfigure}{.5\textwidth}
		\includegraphics[scale=0.75]{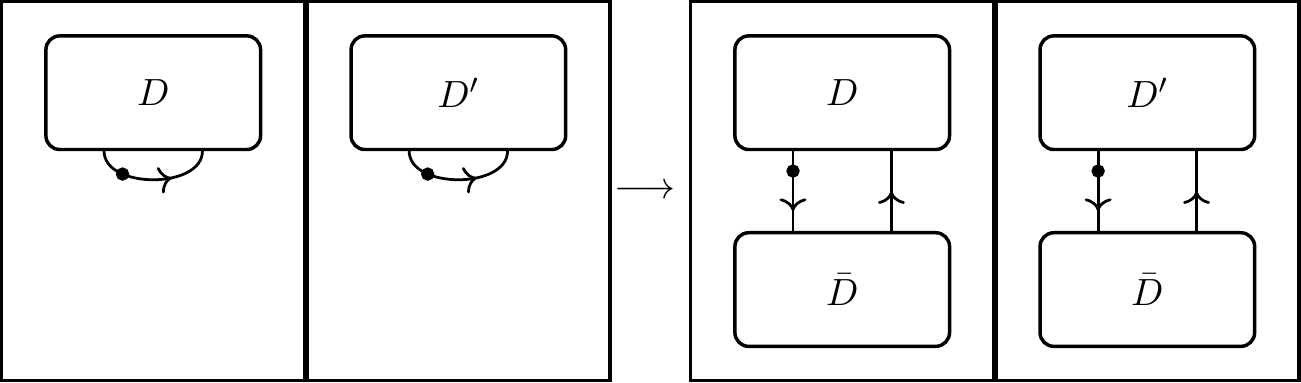}
		\caption{Movies vs connected sum}
		\label{fig:sub2}
	\end{subfigure}
	\caption{Connected sums of pointed links and naturality.}
	\label{fig:test}
\end{figure}

\begin{lemma}\label{CSiso}
	There is an isomorphism with bi-degree $(0,0)$
	\[
	S^{\ast}: \bigoplus_{\substack{i_1+i_2=i \\ j_1+j_2=j}} \khr^{i_1,j_1}(D)\otimes \khr^{i_2,j_2}(D') \longrightarrow \khr^{i,j}(D\sharp D').
	\]
	Moreover we have the following commutative diagram

	\begin{center}
		\includegraphics{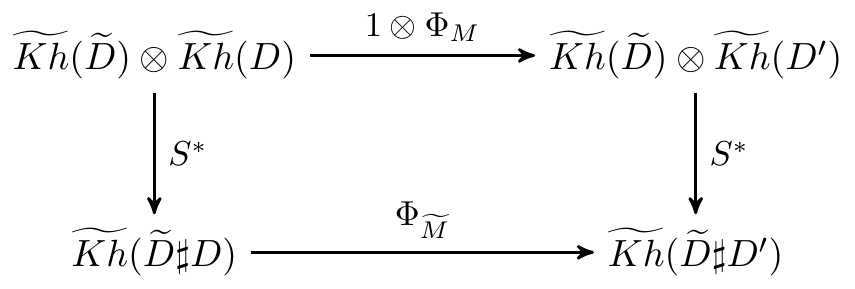}
	\end{center}
	where $\Phi_M$ and $\Phi_{\widetilde{M}}$ are induced by the movies in Figure \ref{fig:sub2}.
\end{lemma}

Note that the result is true for any choice of basepoints $p$ and $p'$ for $D$ and $D'$ respectively.
\medskip
\begin{proof}	
	We start with an explicit description of the chain complex for $D\sharp D'$. First notice that $\chi_{D \sharp D'} = \chi_D \cup \chi_{D'}$. Let $A_{\sharp} \subset \chi_{D \sharp D'}$. Such a subset is in 1-1 correspondence with a pair $(A \subset \chi_D, A' \subset \chi_{D'})$ and one sees easily that 
	\[
	\Lambda V_{ A_ {\sharp}} = \Lambda V_A \wedge \Lambda V_{A'} \mbox{ with the relation } x_{\bullet} = x_{\gamma},
	\]
	where $\gamma$ is labelled in Figure \ref{fig:sub1}. Thus for $y \in \Lambda V_{ A_ {\sharp}} \subset (x_{\bullet}{C}(D \sharp D'),d_{\sharp})$ we write 
	\[
	y = vx_{\bullet}w,
	\]
	with $v \in \Lambda V_A, w \in \Lambda V_{A'}$ containing neither $x_{\bullet}$ nor $x_{\gamma}$. The differential $d_{\sharp}$ can also be made explicit. Denote by $d_D$ (resp. $d_{D'}$) be the differential of the complex $\widetilde{C}(D)$ (resp. $\widetilde{C}(D')$). If we add a crossing c to $A_{\sharp}$, we add a crossing to either $A$ or $A'$. The differential $d_{\sharp}$ splits into two pieces: in the first we sum over all crossing that can be added to $A$, in the second over all those that can be added to $A'$. Therefore 
	\[
	d_{\sharp}(y)= d_D(vx_{\bullet})w + vd_{D'}(x_{\gamma}w) = v'x_{\bullet}w + vx_{\bullet}w',
	\]
	where $d_D(vx_{\bullet}) = x_{\bullet}v'$ and $ d_{D'}(x_{\gamma}w)= x_{\gamma}w'$.
	Let $\partial_{\sharp}:x_{\bullet}C^{i_1,j_1}(D) \otimes C^{i_2,j_2}(D') \longrightarrow x_{\bullet}C^{i, j-1}(D \sharp D')$ be the map induced by the $1$-handle in Figure \ref{fig:sub1}. Since anywhere in chain complex, it fuses the pointed circle with the circle labelled $\gamma$, we have
	\[
	\partial_{\sharp}(x_{\bullet}v \otimes w)= vx_{\bullet}w.
	\] 
	Consider the composite chain map
	\[
	S:= \partial_{\sharp} \circ (1 \otimes \nu):  \bigoplus_{\substack{i_1+i_2=i \\ j_1+j_2=j}} x_{\bullet}C^{i_1, j_1}(D) \otimes  x_{\gamma}C^{i_2, j_2}(D') \longrightarrow x_{\bullet}C^{i, j+1}(D \sharp D').
	\]
	More explicitly we have $S(vx_{\bullet}\otimes x_{\gamma}w)= vx_{\bullet}w$. The chain map $S$ is an isomorphism with inverse $h:\widetilde{C}(D \sharp D') \rightarrow \widetilde{C}(D) \otimes \widetilde{C}(D')$ defined by $h(vx_{\bullet}w)=vx_{\bullet} \otimes x_{\gamma}w$. There is one remaining detail to check: the compatibility of $S$ with the shifts we added in the definition of the reduced chain complex, i.e that $S$ has bi-degree $(0,0)$. Recall that we had
	\[
	\widetilde{C}^{\ast, \ast}(D,p) =  x_{\bullet}C^{\ast,\ast}(D)[0,1].
	\]
	We shift each of $x_{\bullet}C^{i_1, j_1}(D)$ and $x_{\gamma}C^{i_2, j_2}(D')$ by $[0,1]$, i.e. a total shift of $[0,2]$ on $x_{\bullet}C^{i, j+1}(D \sharp D')$. By definition of the shifts, we have
	\[
	x_{\bullet}C^{i, j+1}(D \sharp D')[0,2] = x_{\bullet}C^{i, j}(D \sharp D')[0,1].
	\]
	It follows that $S$ induces a bi-degree $(0,0)$ isomorphism as claimed.
	
	\medskip
	
	For the naturality statement, recall that a map of reduced chain complexes $\Phi_M$ send an element $v = x_{\bullet}w$ to $x_{\bullet}\Phi_M(w)$ so the image of $v$ is completely determined by $\Phi_M(w)$. We show that the diagram commutes at the level of chain complexes. Indeed we have
	\[
	\begin{array}{l}
	S\circ (1 \otimes \Phi_M)(vx_{\gamma} \otimes x_{\bullet}w)=S^{\ast} (vx_{\gamma} \otimes x_{\bullet}\Phi_M(w) )= vx_{\bullet}\Phi_M(w). \\
	\Phi_{\widetilde{M}} \circ S^{\ast} (vx_{\gamma} \otimes x_{\bullet}w) = 	\Phi_{\widetilde{M}} (vx_{\bullet}w) =
	vx_{\bullet}\Phi_M(w), \\
	\end{array}
	\]
	where the last equality is due to the fact that the variables in $v$ arise from circles which are fixed by the movie $M$.
\end{proof}

From this point onwards, we won't mention any basepoints since everything we use is independent from their choice, courtesy of the previous lemmas.
We now turn to the main computational tool, namely the long exact sequence associated to a choice of crossing.

\begin{definition}
	Let $D$ be an oriented diagram, and $c$ a crossing in $D$. Denote by $D_0$ (resp. $D_1$) the diagram obtained by $0$-smoothing (resp. $1$-smoothing) $c$. The triple of diagrams ($D_1,D,D_0$) will be called \emph{the exact triple associated to $c$}.
\end{definition}
To any such exact triple, one can associate a short exact sequence of ungraded chain complexes:
\[
0 \longrightarrow \widetilde{C}(D_1) \overset{i}{\longrightarrow} \widetilde{C}(D) \overset{q}{\longrightarrow} \widetilde{C}(D_0) \longrightarrow 0.
\]

One can introduce gradings as follows. If $c$ is a negative crossing, $D_1$ inherits an orientation from $D$. Choose any orientation for $D_0$ and set $w_{-} = n_{-}(D_0) - n_{-}(D)$, $w_{+}=0$. If $c$ is positive, $D_0$ inherits an orientation from $D$. Choose any orientation for $D_1$ and set $w_{-}=0$, $w_{+}= 1 + n_{-}(D_1) - n_{-}(D)$.
Any exact triple gives rise to short exact sequences of chain complexes. More precisely, for each $j \in \bbZ$ we have:
\[
0 \longrightarrow \widetilde{C}^{\ast,j}(D_1)[w_{+},3w_{+}-1] \overset{i}{\longrightarrow} \widetilde{C}^{\ast,j}(D) \overset{q}{\longrightarrow} \widetilde{C}^{\ast,j}(D_0)[w_{-},3w_{-}+1] \longrightarrow 0
\]

Since $\khr$ is a functor, it is natural to ask oneself whether this short exact sequence is also functorial. Such property is standard when considering the \emph{ungraded} chain complexes. However, for the graded version, one still needs to check that all maps respect the various grading shifts. First, we remark that the two maps induced on $\widetilde{C}^{\ast,j}(D_1)$ and $\widetilde{C}^{\ast,j}(D_0)$ are obtained from the movie starting at $D$ ending at $D'$ by replacing the crossing $c$ by its $1$ and $0$ smoothing respectively. Note that one of these might not be an oriented $1$-handle move, as it depends on a choice of orientation.

\begin{lemma}
	Let $D,D'$ be two diagrams related by a $1$-handle move. For any choice of crossing $c$, the map induced by the move also induces a map of the associated short exact sequences.
	\[
	\begin{array}{ccccccc}
	0 \longrightarrow &\widetilde{C}^{\ast,j}(D_1)[w_{+},3w_{+}-1]& \overset{i}{\longrightarrow}& \widetilde{C}^{\ast,j}(D)& \overset{q}{\longrightarrow}& \widetilde{C}^{\ast,j}(D_0)[w_{-},3w_{-}+1]& \longrightarrow 0 \\
	& \Big\downarrow \Phi^1 && \Big\downarrow \Phi && \Big\downarrow \Phi^0& \\
	0 \longrightarrow &\widetilde{C}^{\ast,j-1}(D'_1)[w'_{+},3w'_{+}-1]& \overset{i}{\longrightarrow}& \widetilde{C}^{\ast,j-1}(D')& \overset{q}{\longrightarrow}& \widetilde{C}^{\ast,j-1}(D'_0)[w'_{-},3w'_{-}+1]& \longrightarrow 0. \\
	\end{array}
	\]
\end{lemma}

\begin{proof}
	We only need to check that $\Phi^1$ and $\Phi^0$ have the proper degrees. No matter the sign of the crossing, if both induced movies, i.e. where $c$ is replaced by its two smoothings, can be made into an oriented $1$-handle move by some choice of orientations then $n_{-}(D)=n_{-}(D')$, $n_{-}(D_1)=n_{-}(D_1')$, $n_{-}(D_0)=n_{-}(D_0')$ and it follows that $w_{+}=w_{+}'$ and $w_{-}=w_{-}'$. Thus the induced chain map has the form, for $m=0,1$:
	\[
	\Phi^m:\widetilde{C}^{\ast,j}(D_m) \longrightarrow \widetilde{C}^{\ast,j-1}(D_m').
	\]
	Both appear in the short exact sequence as shifted versions, with domain and co-domain consistently shifted. 
	
	Let us assume that $c$ is negative and that the induced movie for the $0$-smoothing cannot be oriented. First remark that $D_1$ and $D_1'$ inherit the orientations so the induced movie is always oriented and treated as above. The two strands of $D_0$ in the the $1$-smoothing of $c$ must point in the same direction. We consider an intermediate $\bar{D}$, identical to $D_0$ except we replace the $0$-smoothing of $c$ by a positive crossing $\bar{c}$. This produces an exact triple $(\bar{D}_1 = D_0',\bar{D},D_0)$ associated to $\bar{c}$. Let $\bar{w}_{+}$ the corresponding shift. We then have the following sequence of equalities:
	\[
	\begin{array}{ccl}
	\bar{w}_{+} &=& n_{-}(\bar{D}_1) - n_{-}(\bar{D}) + 1 \\
				&=& n_{-}(D_0') - n_{-}(D_0) + 1 \\
				&=& n_{-}(D_0') - n_{-}(D') + n_{-}(D) - n_{-}(D_0) + 1 \\
				&=& w_{-}' - w_{-}+1.
	\end{array}
	\]
	The second equality follows from $\bar{D}_1 = D_0'$ and $n_{-}(\bar{D})=n_{-}(D_0)$ - since $D_0$ and $\bar{D}$ differ only by a positive crossing $\bar{c}$. The third uses $n_{-}(D)=n_{-}(D')$ and the last one follows from the definition of $w_{-}'$ and  $w_{-}$. We now have a graded map 
	\[
	\Phi^0:\widetilde{C}^{\ast,j}(D_0)[0,1] \longrightarrow \widetilde{C}^{\ast+1,j}(D_0')[\bar{w}_{+},3\bar{w}_{+}-1].
	\]
	The previous computation of $\bar{w}_{+}$ yields 
	\[
	[\bar{w}_{+},3\bar{w}_{+}-1] = [ w_{-}' - w_{-}+1,3(w_{-}' - w_{-}+1)-1] = [ w_{-}' - w_{-}+1,3w_{-}' - 3w_{-} +2].
	\]
	There are two steps left. First, we shift both chain complexes by $[w_{-},3w_{-}]$:
	\[
	\Phi^0:\widetilde{C}^{\ast,j}(D_0)[w_{-},3w_{-}+1] \longrightarrow \widetilde{C}^{\ast+1,j}(D_0')[ w_{-}'+1,3w_{-}'+2].
	\]
	Finally, we apply the definition of the shift to change $\widetilde{C}^{\ast+1,j}(D_0')$ into $\widetilde{C}^{\ast,j-1}(D_0')$. This yields:
	\[
	\Phi^0:\widetilde{C}^{\ast,j}(D_0)[w_{-},3w_{-}+1] \longrightarrow \widetilde{C}^{\ast,j-1}(D_0')[ w_{-}',3w_{-}'+1].
	\]
	The map $\Phi^0$ has the proper degree, and we get a map relating the two graded short exact sequences as claimed. The positive case can be treated similarly.
	
\end{proof}
From this short exact sequence, we derive a long exact sequence for each quantum degree $j$. It provides us with a collection of boundary maps, one for each $i \in \bbZ$:
\[
\{ \partial_j: \khr^{i,j}(D_0)[w_{-},3w_{-}+1] \longrightarrow \khr^{i+1,j}(D_1)[w_{+},3w_{+}-1]\}_{i \in \bbZ}
\]
From this data, one can create a chain complex, whose homology is $\khr^{\ast,j}(D)$, and given by
\[
V(D,c,j) := \bigoplus_{\substack{i \in \bbZ}}\khr^{i,j}(D_0)[w_{-},3w_{-}+1]  \bigoplus  \khr^{i+1,j}(D_1)[w_{+},3w_{+}-1], \hspace{0.2cm} d_j= \begin{pmatrix}
0 & \oplus \partial^i_j \\
0 & 0
\end{pmatrix}.
\]
\begin{remark}
	This chain complex is just the $E_1$-page of Turner's spectral sequence \cite{turner2008spectral}. In accordance to this observation, we will call the boundary map $d_j$ a "differential".
\end{remark}

We can add one extra step which is the sum of these chain complexes, over all $j \in \bbZ$. This fact motivates the definition below.

\begin{definition}
	Let $D$ be an oriented diagram and $c$ a crossing of $D$. The \emph{total exact sequence} for $D$ with respect to $c$ is the chain complex whose underlying vector space is
	$V(D,c)= \bigoplus\limits_{j \in \bbZ} V(D,c,j)$ with differential $d=\bigoplus\limits_{j \in \bbZ} d_j$.
	We will say that the total exact sequence \emph{abuts to} or \emph{converges to} $\khr^{\ast,\ast}(D)$.
\end{definition}
Obviously, this total exact sequence fits in a computational framework, so we will mainly consider the $\delta$-graded version
\[
\khr^{\ast}_{\ast}(D_0)[w_{-},w_{-}+1] \bigoplus \khr^{\ast+1}_{\ast}(D_1)[w_{+},w_{+}-1]
\]
converging $\khr^{\ast}_{\ast}(D)$.
We will present it as a table, with homologically graded columns and $\delta$-graded rows. Each entry ($i,\delta$) of the table will be filled with integers indexed by either $0$ or $1$. These integers are the dimensions of $\khr^{i}_{\delta}(D_0)[w_{-},w_{-}+1]$ if the index is $0$, and of $\khr^{i}_{\delta}(D_1)[w_{+},w_{+}+1]$ if the index is $1$. The differential of the total sequence has bi-degree $(1,-2)$ and we will indicate the positions in which it is possibly non-trivial by an arrow from an entry $(i,\delta)$ to an entry $(i+1, \delta-2)$. Before we move on to the main object of this work, let us give an example of use of the total exact sequence. Note that in this example we won't compute a homology but rather reverse-engineer a total exact sequence.

\begin{example}\label{T332q}
	 Let us first introduce a family of links. We call $T_{(3,3),(2,q)}$ the closure of the middle braid below (with the circled crossing), where $q$ denotes the number of negative half-twist on top. These braids will always be equipped with the upwards orientation. In this example, we consider the torus link $T_{3,3}=T_{(3,3),(2,0)}$. Clearly $n_{-}(D)= 6$. Let $c$ be the circled crossing. We obtain an exact triple of diagrams.
	 \begin{center}
	 \includegraphics[scale=0.65]{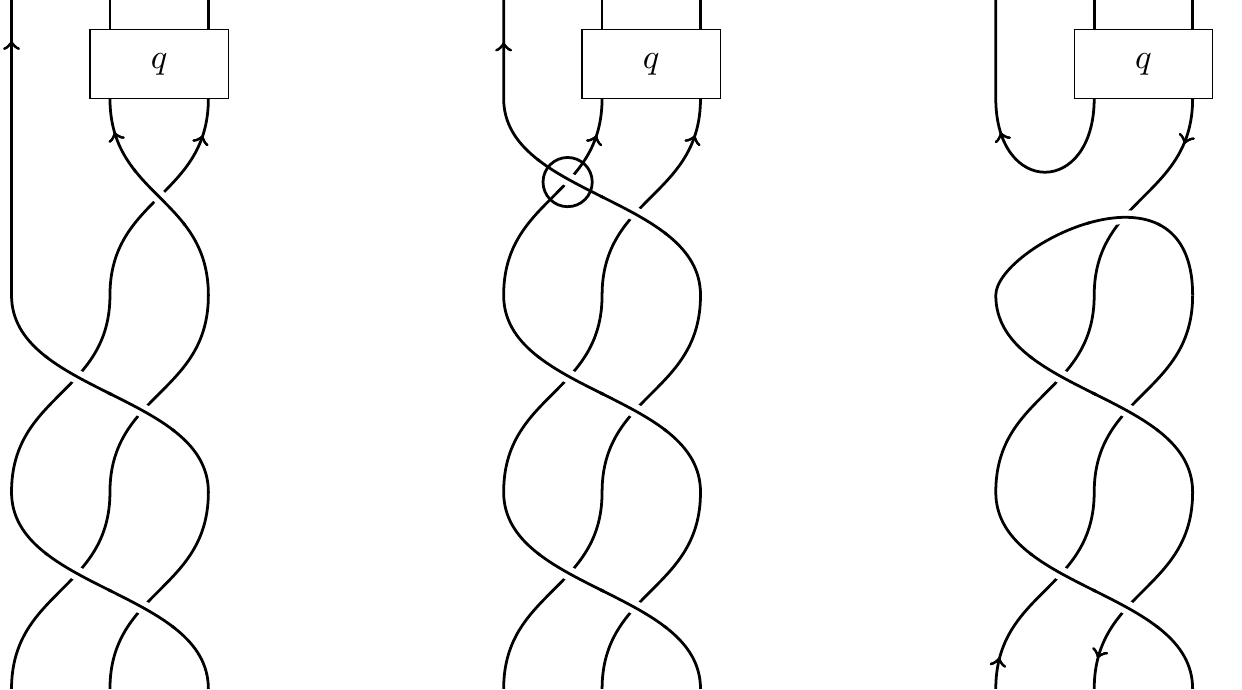}
	\end{center} 
	 One checks easily that for $q=0$, we have an exact triple:
	 \[
	 (D_1, D, D_0) = (T_{2,4}, T_{3,3}, U \coprod U).
	 \]
	 With the orientations prescribed above, $n_{-}(D_0)=2$ thus we have $w_{-}=2-6=-4$ and $w_{+}=0$. Hence our $\delta$-graded total sequence corresponds to the grid below (on the left), while the grid on the right gives $\khr^{\ast}_{\ast}(T_{3,3})$.
	 \begin{center}
	 \includegraphics[scale=0.75]{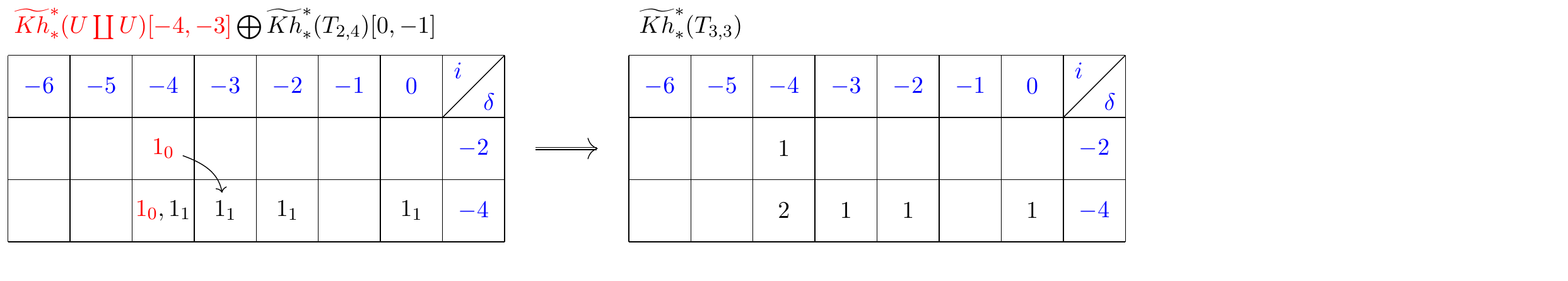}
	 \end{center}
	 It follows that the only possibly non-trivial differential (the arrow in the left hand grid) is in fact trivial.
\end{example}

\subsection{Direct limits in Khovanov homology}
In order to construct the stable homology of torus links, we need to recall some facts about directed systems of chain complexes.

\begin{definition}
	A \emph{directed system of chain complexes} $ \langle V_i,f_{ji} \rangle$ is a family $\{ V_i \}$ of chain complexes, indexed by a partially ordered set $(I, \leq)$, together with chain maps $f_{ji}:V_i \rightarrow V_j $ for each $i \leq j$ such that
	\begin{itemize}
		\item {$f_{ii}$ is the identity of $V_i$.}
		\item {$f_{ki} = f_{kj}  \circ f_{ji}$ for all $i \leq j \leq k$.}
	\end{itemize} 
	To any such system, one can associate another chain complex $V$ defined by 
	\[
	V = \bigoplus V_i \textfractionsolidus \sim.
	\]
	The relation $\sim$ is defined as follows: if $x_i \in V_i, x_j \in V_i$, we have $x_i \sim x_j$ if there exists $k$ such that $f_{ki}(x_i)=f_{kj}(x_j)$. The complex $V$ is called \emph{direct limit}, denoted by $\underset{\longrightarrow}{\lim} \hspace{0.1cm} V_i$. Note that each $V_i$ is equipped with a map $\varphi_i:V_i \rightarrow V$, the \emph{canonical projection}, sending an element to its equivalence class.	
\end{definition}

The direct limit is a functor from the (monoidal) category of directed systems of chain complexes to the (also monoidal) category of chain complexes. Indeed, let  $ \langle V_i,f_{ji} \rangle $ and  $ \langle V'_i,f'_{ji} \rangle $ be two directed systems, and $u_i : V_i \rightarrow V'_i$ a family of chain maps verifying
$f'_{ji} \circ u_i = u_j \circ f_{ji}$, then there is an induced chain map 
\[
u: \underset{\longrightarrow}{\lim} \hspace{0.1cm} V_i \rightarrow \underset{\longrightarrow}{\lim} \hspace{0.1cm} V'_i, \mbox{ such that } \varphi_i'\circ u_i = u \circ \varphi_i
\]
Let us mention some selected properties of that functor.

\begin{proposition}\label{limit}\cite{bourbaki1970algebre}
	
	The $\underset{\longrightarrow}{\lim}$ functor has the following properties
	\begin{enumerate}[(i)]
		\item { The functor $\underset{\longrightarrow}{\lim}$ is exact.
		}
		\item{ The functors $\underset{\longrightarrow}{\lim}$ and $\otimes_{\bbZ_2}$ commute.
		}
		\item{ The functors $\underset{\longrightarrow}{\lim}$ and $H_{\ast}$ commute.
		}		
	\end{enumerate}	
\end{proposition}

Let $D_{p,q}$ be the closure of the standard braid diagram for the negative torus link $T_{p,q}$, with upwards orientation.
For each $p \geq 2$, one can consider the sequence of diagrams pictured below 

\begin{center}
	\includegraphics[scale=0.75]{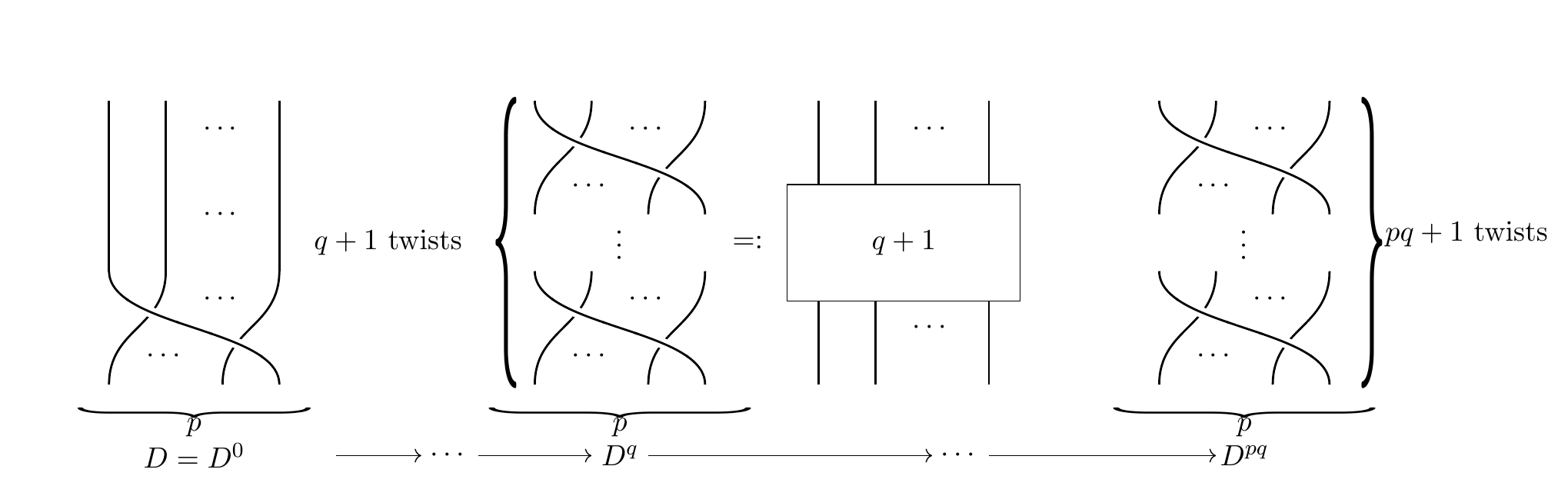}
\end{center}

There is a corresponding sequence of inclusions of (quantum graded) chain complexes:
\[
\cdots \overset{i_{q-1}}{\hookrightarrow} \widetilde{C}(D_{p,q})[0,(p-1)(q-1)] \overset{i_q}{\hookrightarrow} \widetilde{C}(D_{p,q+1})[0,(p-1)q] \overset{i_{q+1}}{\hookrightarrow} \widetilde{C}(D_{p,q+2})[0,(p-1)(q+1)] \overset{i_{q+2}}{\hookrightarrow} \cdots
\]
where $i_q$ is induced by $1$-smoothing top row of $D_{p,q+1}$, producing the diagram $D_{p,q}$. Note that the chain complexes have been shifted so that the inclusions all have (quantum) bi-degree (0,0). We can now form the obvious directed system of chain complexes:
\[
\langle \widetilde{C}(D_{p,q})[0,(p-1)(q-1)] , f_{jq} \rangle, \mbox{ where } f_{jq} = i_{j} \circ i_{j-1} \circ \cdots \circ i_{q} \mbox{ if } q < j, \mbox{ and } f_{qq} = 1.
\]
At the limit, this process yields a chain complex $(\widetilde{C}(D_{p,\infty}), d_{\infty})$ and its homology is denoted by $\widetilde{Kh}^{\ast,\ast}(T_{p,\infty})$. Though this point of view is useful to discuss properties of $\widetilde{Kh}^{\ast,\ast}(T_{p,\infty})$, for obvious reasons of dimension, it is unsuited to computations. More practical is the induced directed system of homologies:
\[
\cdots \overset{i^{\ast}_{q-1}}{\hookrightarrow} \widetilde{Kh}(T_{p,q})[0,(p-1)(q-1)] \overset{i_q}{\hookrightarrow} \widetilde{Kh}(T_{p,q+1})[0,(p-1)q] \overset{i^{\ast}_{q+1}}{\hookrightarrow} \widetilde{Kh}(T_{p,q+2})[0,(p-1)(q+1)] \overset{i^{\ast}_{q+2}}{\hookrightarrow} \cdots
\]
By property ($iii$) of Proposition \ref{limit}, one has the relation 
\[
\widetilde{Kh}^{\ast,\ast}(T_{p,\infty}) = \underset{\longrightarrow}{\lim}\hspace{0.1cm} \widetilde{Kh}^{\ast,\ast}(T_{p,q})[0,(p-1)(q-1)].
\]
There is a corresponding vector space for the $\delta$-graded version, $\widetilde{Kh}^{\ast}_{\ast}(T_{p,\infty})$.
Let us begin our study with the case $p=2$.

\begin{example}\label{2strandedlimit}
The vector space $\widetilde{Kh}^{\ast}_{\ast}(T_{2,\infty})$ can be explicitly described as
	\[
	\widetilde{Kh}^{i}_{\delta} (T_{2,\infty}) = 
	\begin{dcases*}
	\bbZ_{2} & if $\delta=0 $ and $i \in \{0,-2,-3,-4,\ldots \}$ \\
	0 & otherwise.
	\end{dcases*}
	\]
It is well known that for fixed $q \geq 2$, we have a description 
	\[
	\widetilde{Kh}^{i}_{\delta} (T_{2,q}) = 
	\begin{dcases*}
	\bbZ_{2} & if $\delta=-q+1 $ and $i \in \{0,-2,-3,\ldots,-q \}$ \\
	0 & otherwise.
	\end{dcases*}
	\]
Moreover, the shifted inclusions $i_q :\widetilde{Kh}^{i}_{\delta} (T_{2,q})[0,q-1] \longrightarrow \widetilde{Kh}^{i}_{\delta} (T_{2,q+1})[0,q]$ are all injective and have bi-degree $(0,0)$. Thus $\khpqd{2}{\infty}$ is as claimed. Finally the projections
\[
\varphi_q:\widetilde{Kh}^{i}_{\delta} (T_{2,q})[0,q-1] \longrightarrow \widetilde{Kh}^{i}_{\delta} (T_{2,\infty})
\]
are also all injective.
\end{example}

For $p=2$, the stable space is non-trivial. This observation has been generalized for all $p$ by Sto\v{s}i\'c. More precisely, he showed the following theorem, whose statement we adapt to our use of negative torus links.

\begin{theorem}\cite{stovsic2007homological}
	Let $p,q,i$ be integers such that $2 \leq p < q$ and $i > 3 - p - q$. Then for any $j \in \bbZ$, there is an isomorphism
	\[
	H^{i,j}(D_{p,q})  \cong 	H^{i,j}(D_{p,q-1}),
	\] where $H^{\ast,\ast}(D)$ is the unnormalized Khovanov homology of the diagram $D$, and $D_{p,q}$ is the standard diagram for $T_{p,q}$.
\end{theorem}

In order to understand these vector spaces, we will upgrade them to algebras, through the use of movies made of successive $1$-handles.

\section{The algebra structures}

This section's main focus is the definition of an algebra structure for $\khpq{p}{\infty}$ and its explicit description for the $2$-stranded case. In a second part, we associate to tangles some module structures over (some) $\khpq{p}{\infty}$.

\subsection{Definition of the algebra structures}

Consider two braids $\be, \be'$ over $p$ strands. There is an obvious movie starting at the connected sum of their closures $\hat{\be}, \hat{\be'}$ and ending at the closure of their composition (as braids). For example, if $p=3$, and $\hat{\be}=T_{3,q},\hat{\be'}=T_{3,q'}$, then $\widehat{\be\circ \be'}= T_{3,q+q'}$ and the movie is pictured in Figure \ref{FusionCob}.
\begin{center}
	\includegraphics{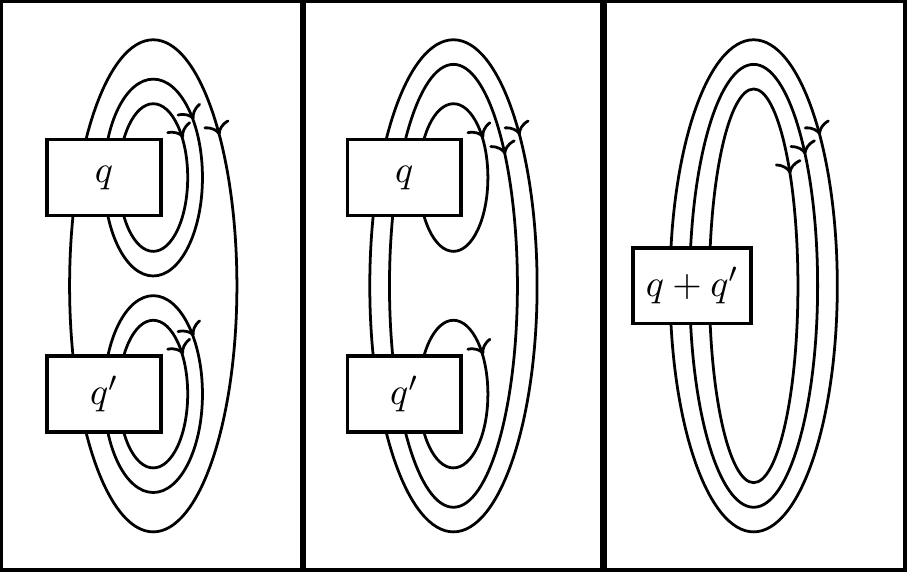}
	\captionof{figure}{The fusion movie for $p=3$}\label{FusionCob}
\end{center}

For any $p \geq 2$, the vector space $\widetilde{Kh}(T_{p,\infty})$ can be endowed with a structure of algebra, which we describe now. Recall that any movie made of successive $1$-handles produces a chain maps $\Phi^1, \Phi$ such that the following diagram commutes:
\[
\begin{array}{ccc}
\widetilde{C}(D_1) & \longrightarrow & \widetilde{C}(D) \\
\big\downarrow \Phi^1 & & \big\downarrow \Phi \\
\widetilde{C}(D'_1) & \longrightarrow & \widetilde{C}(D')
\end{array}
\]
where $D_1$ (resp. $D'_1$) is obtained from $D$ (resp. $D'$) by $1$-smoothing of a crossing $c$ and the horizontal arrows are the inclusions of the corresponding sub-complexes.
Define 
\[
\Sigma^p_{q,q'}: \widetilde{C}(D_{p,q} \sharp D_{p,q'}) \longrightarrow \widetilde{C}(D_{p,q+q'})
\]
as the map of chain complexes with bi-degree $(0,-p + 1)$ induced by the fusion movie over $p$ strands. Thus in homology we get a commutative ladder
\[
\begin{array}{ccccc}
\cdots \overset{i^{\ast}_{q-1}}{\longrightarrow} & \widetilde{Kh}(T_{p,q} \sharp T_{p,q'})& \overset{i_q}{\longrightarrow} & \widetilde{Kh}(T_{p,q+1}\sharp T_{p,q'})& \overset{i^{\ast}_{q+1}}{\longrightarrow} \cdots \\
& \big\downarrow \Sigma^p_{q,q'}& & \big\downarrow \Sigma^p_{q+1,q'} & \\
\cdots \overset{i^{\ast}_{q+q'-1}}{\longrightarrow} & \widetilde{Kh}(T_{p,q+q'})& \overset{i_{q+q'}}{\longrightarrow} & \widetilde{Kh}(T_{p,q+q'+1})& \overset{i^{\ast}_{q+q'+1}}{\longrightarrow} \cdots \\
\end{array}
\]
There is a similar ladder if $q'$ increases instead of $q$. Now, consider the composite
\[
\widetilde{\Sigma}^p_{q,q'}: \khr^{i,j}(T_{p,q}) \otimes \khr^{i',j'}(T_{p,q'}) \overset{S^{\ast}}{\longrightarrow} \khr^{i+i', j+j'}(T_{p,q} \sharp T_{p,q'}) \overset{\Sigma^p_{q,q'}}{\longrightarrow} \khr^{i+i',j+j'-p+1}(T_{p,q+q'})
\]

In order to match the directed system which produces $\widetilde{Kh}^{\ast,\ast}(T_{p,\infty})$, let us shift $\widetilde{Kh}^{\ast,\ast}(T_{p,q})$ by $[0,(p-1)(q-1)]$ and $\widetilde{Kh}^{\ast,\ast}(T_{p,q'})$ by $[0,(p-1)(q'-1)]$, thus performing a total shift of $[0,(p-1)(q+q'-2)]$. If we shift the co-domain accordingly, by definition of the shifts we have
\[
\khr^{i+i',j+j'-p+1}(T_{p,q+q'})[0,(p-1)(q+q'-2)]=\khr^{i+i',j+j'}(T_{p,q+q'})[0,(p-1)(q+q'-1)].
\]
Hence the shifted version of the map $\widetilde{\Sigma}^p_{q,q'}$ has degree $(0,0)$:

\begin{center}
	\includegraphics{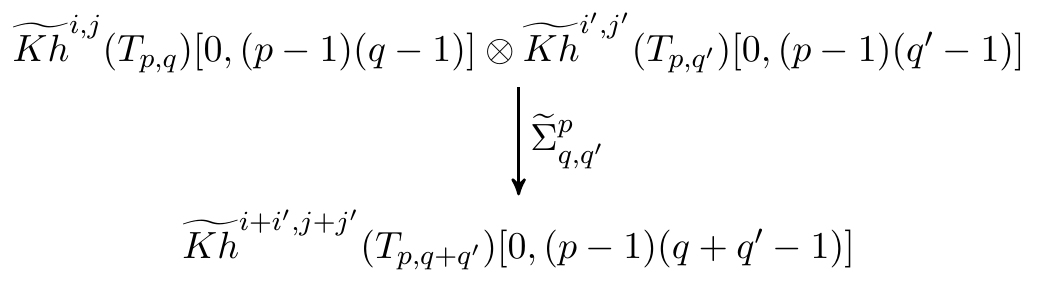}
\end{center}

Let us briefly use the notation $A_q := \khr^{\ast,\ast}(T_{p,q+1})[0,(p-1)(q-1)]$ and consider the following diagram, where all squares and triangles commute. For the horizontal faces, commutativity follows from the definition of the maps involved. For the vertical faces of the cubes where one parameter ($q$ or $q'$) is fixed, commutativity follows from that of the previous ladders. Hence the last face, where both parameters $q$ and $q'$ increase, also commutes.

\begin{center}
	\includegraphics{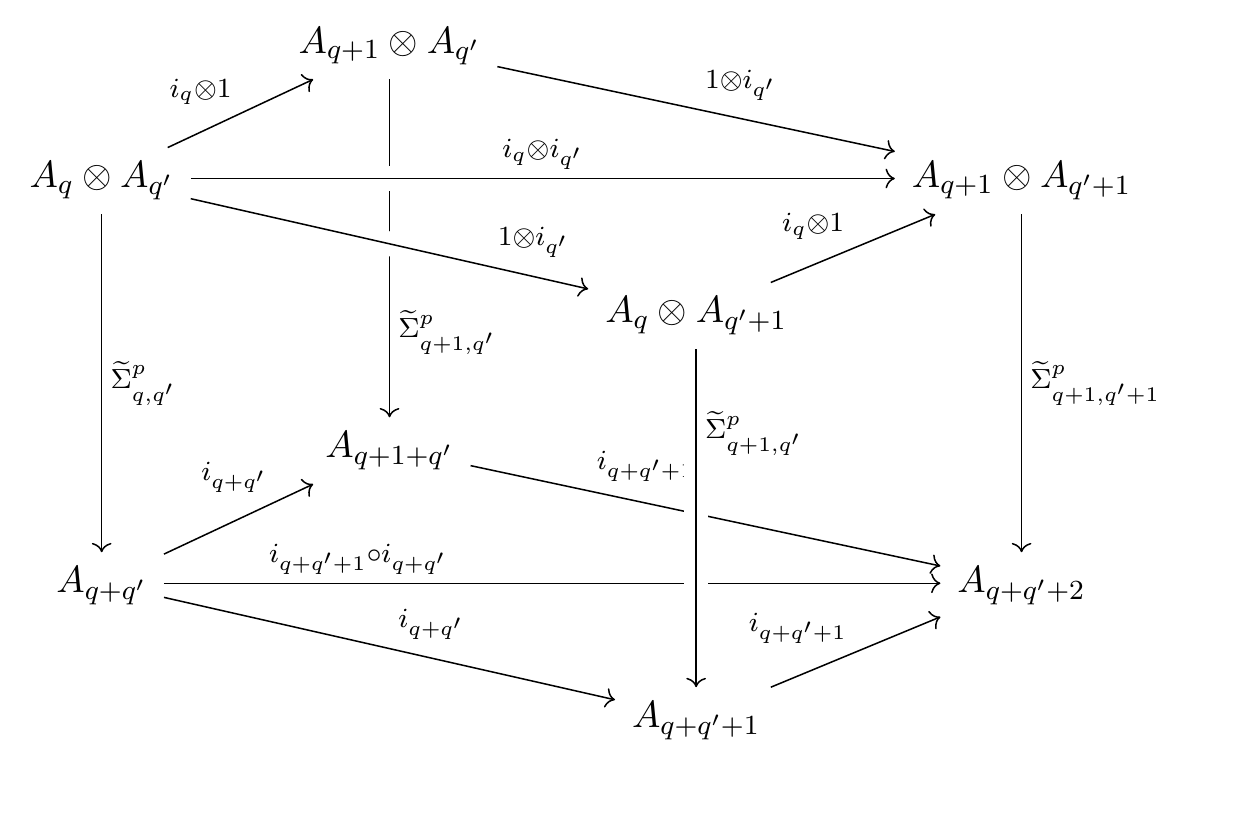}
\end{center}

The commutativity makes $\widetilde{\Sigma}^p_{q,q'}$ into a map of directed systems with degree (0,0):
\[
\widetilde{\Sigma}^p_{q,q'}: \langle A_{q} \otimes A_{q'}, i_{q} \otimes i_{q'} \rangle \longrightarrow \langle A_{q+q'}, i_{q+q'+1} \circ i_{q+q'}\rangle
\] 
The direct limit functor, together with property $(ii)$ of Proposition \ref{limit}, produces a map:
\[
\widetilde{\Sigma}^p_{\infty}: \khr^{i,j}(T_{p,\infty}) \otimes \khr^{i',j'}(T_{p,\infty}) \longrightarrow \khr^{i+i',j+j'}(T_{p,\infty}),
\]

which satisfies the relation $\widetilde{\Sigma}^p_{\infty} \circ (\varphi_q \otimes \varphi_{q'}) = \varphi_{q+q'} \circ \widetilde{\Sigma}^p_{q,q'}$. 

We will refer to the maps $\widetilde{\Sigma}^p_{\infty}$ and their finite counterparts $\widetilde{\Sigma}^p_{q,q'}$ as \emph{fusion maps}. We are now ready to state and prove the main theorem of this section, namely that the fusion maps give additional structure to the stable spaces. Given $a,b \in \khpq{p}{\infty}$, define their \emph{product} to be $a \cdot b := \widetilde{\Sigma}^p_{\infty}(a \otimes b)$.

\begin{theorem}\label{algstruc}
	For each $p\geq 2$, the fusion map $\widetilde{\Sigma}_{\infty}^p: \khpq{p}{\infty} \otimes \khpq{p}{\infty} \longrightarrow \khpq{p}{\infty}$ endows the stable space $\khpq{p}{\infty}$ with a structure of commutative bi-graded unitary algebra.
\end{theorem}

\begin{proof}
	First remark that all maps involved are linear so any distributivity property will hold. The associativity follows from Jacobsson's result (\cite{jacobsson2004invariant}, Theorem 2). This theorem states that two movies which differ only by an exchange of distant critical points induce maps that agree up to sign (for integer coefficients). With $\bbZ_2$ coefficients, two such maps must agree. Therefore, for any choice of $q,q',q''$, we have an equality
	\[
	\widetilde{\Sigma}^p_{q,q'+q''} \circ (1 \otimes \widetilde{\Sigma}^p_{q',q''}) =	\widetilde{\Sigma}^p_{q+q',q''} \circ (\widetilde{\Sigma}^p_{q,q'} \otimes 1),
	\]
	that translates through the $\underset{\longrightarrow}{\lim}$ functor into:
	\[
	\widetilde{\Sigma}^p_{\infty} \circ (1 \otimes \widetilde{\Sigma}^p_{\infty}) = \widetilde{\Sigma}^p_{\infty} \circ (\widetilde{\Sigma}^p_{\infty} \otimes 1).
	\]
	Thus we have associativity of our product. Also, since the map $\widetilde{\Sigma}_{\infty}^p$ has bi-degree $(0,0)$, our structure will be bi-graded. We are left with proving the existence of a neutral element and the commutativity. Fix $p\geq 2$ and consider the diagram $D_{p,1}$, the unknot . We work at the chains level in homological degree $0$ since $\khr^{i,j}(D_{p,1})=0$ for $(i,j) \neq (0,0)$. The corresponding smoothing $s_A$ is a collection of $p$ circles. Since nothing depends on the choice of basepoint, we can choose it to lie on the outermost strand, so that it will be fixed by the movie. One checks easily that $\khr^{0,0}(D_{p,1})$ is generated by $x_{\bullet}$. We aim to show that, for any $q \geq 1$, the fusion map
	\[
	\widetilde{\Sigma}^p_{1,q}: \khr^{0,0}(T_{p,1}) \otimes \khr^{i',j'}(T_{p,q}) \longrightarrow\khr^{i',j'-p+1}(T_{p,q+1})
	\]
	sends $x_{\bullet} \otimes x_{\gamma}v$ to $x_{\bullet}v$. Since $S(x_{\bullet} \otimes x_{\gamma}v)=x_{\bullet}v$, we need only show that $x_{\bullet}v$ is fixed by $\Sigma^p_{1,q}$.
	The fusion map is given by a succession of products $m_{A,B}$. The circles involved in these products for the $D_{p,1}$ piece are all equipped with $1 \in \Lambda V_A$. Hence our map acts trivially on $v$. 
	
	Commutativity also follows from a detailed look at the chain complex. We start with the connected sum of two braid closures. Each one is equipped with $p$ closing strands, that we label $1$ to $p$ starting from the outside. Any smoothing inherits this numbering (note that one circle may have multiple labels). The fusion map acts on the strands while preserving the numbering. In order of labels, the map either fuses the circles labelled $i$ if they are different or splits the circle with two labels $i$. The fact that we have either 1 or 2 circles does not depend on the position of our braid closures, thus we have commutativity.
\end{proof} 

\subsection{$2$-stranded torus links}
We study the algebra structure associated to the family of $2$-stranded torus links. We will rely on our explicit knowledge of $\khpqd{2}{\infty}$ (described in Example \ref{2strandedlimit}). The main argument is contained in Lemma \ref{2fusion} whose proof we delay to Section 5. Note that since we are now in the realm of computations, we will use the $\delta$-graded version of Khovanov homology.

\begin{lemma}\label{2fusion}
	For any $q \geq 3$, the fusion maps
	\[
	\begin{aligned}
	\widetilde{\Sigma}_{2,q}^2: \widetilde{Kh}^{\ast}_{\ast}(T_{2,2}) \otimes \widetilde{Kh}^{\ast}_{\ast}(T_{2,q}) \longrightarrow \widetilde{Kh}^{\ast}_{\ast}(T_{2,2+q}), \\
	\widetilde{\Sigma}_{3,q}^2: \widetilde{Kh}^{\ast}_{\ast}(T_{2,3}) \otimes \widetilde{Kh}^{\ast}_{\ast}(T_{2,q}) \longrightarrow \widetilde{Kh}^{\ast}_{\ast}(T_{2,3+q})
	\end{aligned}
	\]
	are surjective.
\end{lemma}

\begin{theorem}\label{Th2inf}
	There is a bi-graded algebra isomorphism:
	\[
	\widetilde{Kh}^{\ast}_{\ast}(T_{2,\infty}) \cong \mathbb{Z}_2[x,y]/(x^3=y^2).	
	\]
	The degrees of the generators are given by $|x|=(-2,0),\hspace{0.25cm} |y|=(-3,0).$
\end{theorem}

\begin{proof}
	Let $a_i$ be the generator of $\khr^{-i}_{\ast}(T_{2,q}) \cong \bbZ_2$. Recall from Example \ref{2strandedlimit} that the projections $\varphi_q: \khpqd{2}{q}[0,q-1] \longrightarrow \khpqd{2}{\infty}$ are injective and hence provide an identification between the $a_i$'s and the generators of $\khpqd{2}{\infty}$.
	As an algebra, there are at least 2 generators $a_2, a_3$ in homological degrees $-2, -3$ respectively.
	Let $n \geq 5$. By Lemma \ref{2fusion}, $a_n$ lies in the image of the map
	\[
	\widetilde{\Sigma}^2_{2,q'}:\khr^{-2}_{\ast}(T_{2,2}) \otimes \khr^{-n+2}_{\ast}(T_{2,n-2}) \longrightarrow \khr^{-n}_{\ast}(T_{2,n}).
	\]
	Moreover $\dim(\khr^{-2}_{\ast}(T_{2,2}) \otimes \khr^{-n+2}_{\ast}(T_{2,n-2})) =1 = \dim(\khr^{-n}_{\ast}(T_{2,n}))$. It follows that the restriction of $\widetilde{\Sigma}^2_{2,q'}$ to these particular degrees is an isomorphism and we have  an identification 
	\[
	a_n = \widetilde{\Sigma}^2_{2,q'}(a_2 \otimes a_{n-2}) = a_2 \cdot a_{n-2}.
	\]
	Furthermore, since $\dim(\khr^{-2}_{\ast}(T_{2,2}) \otimes \khr^{-2}_{\ast}(T_{2,3})) =1 = \dim(\khr^{-4}_{\ast}(T_{2,5}))$, we obtain the relation $a_4 = a_2 \cdot a_2 = a_2^2$.
	The other surjectivity in Lemma \ref{2fusion} can be used similarly and provides identifications $a_n = a_3 \cdot a_{n-3}$. Hence, by induction on $n$, our algebra has exactly $2$ generators, namely $a_2$ and $a_3$.
    Finally, decompose $6 = 2 + 2 + 2 = 3 + 3$ to obtain the relation
	\[
	a_2^3 = a_6 = a_3^2.
	\]	
	The obvious map between $\widetilde{Kh}^{\ast}_{\ast}(T_{2,\infty})$ and $\bbZ_2[x,y]/(x^3=y^2)$ is an algebra isomorphism.
\end{proof}

\subsection{Modules over $\khpq{p}{\infty}$}

The proof of Theorem \ref{Th2inf} relied heavily on the fact that some fusion maps where surjective. However, in general, such property is not easy to obtain. In order to circumvent that difficulty, we associate limits vector spaces to tangle closures and module structures over $\khpq{p}{\infty}$ to these limit spaces. The particular case of modules over $\khpq{p-1}{\infty}$ will play a key role in understanding $\khpq{p}{\infty}$.

\begin{center}
	\includegraphics[scale=0.90]{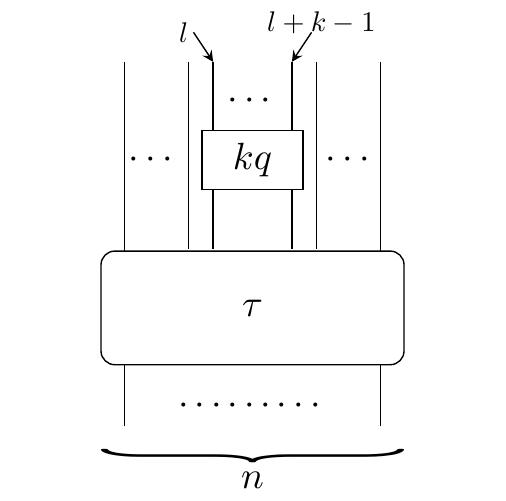}
\end{center}

Let $\tau$ be an oriented $(n,n)$-tangle. We assume that the orientations of the top and bottom strands agree, so that the closure of $\tau$ is naturally oriented. Fix two integers $2\leq k \leq n$ and $1 \leq l \leq n-k + 1$, and denote by $\Delta^k_l$ the negative full twist over $k$ strands starting at the strand $l$. Consider the family of tangles $\tau^q = (\Delta^k_l)^q \circ \tau$, pictured above. All tangles in that family inherit an orientation from $\tau$ since they differ only by a full twist. For each $q \geq 0$, let $D^q$ be the oriented closure of $\tau^q$.  From this family we construct a directed system of chain complexes associated to the sequence of inclusions below.
\[
\cdots \overset{i_{q-1}}{\hookrightarrow} \widetilde{C}(D^q)[0,(k-1)(q(k-1))] \overset{i_q}{\hookrightarrow} \widetilde{C}(D^{q+1})[0,(k-1)(qk)] \overset{i_{q+1}}{\hookrightarrow} \cdots
\]
where the $i_q$'s are obtained by $1$-smoothing the topmost full twist $\Delta^k_l$ of $(\Delta^k_l)^q$. Note that the degrees of the inclusions will match the shifts if and only if all the strands are oriented consistently in the same direction. This fact motivates the definition below.

\begin{definition}
		Let $\tau$ be an oriented $(n,n)$-tangle that closes to an oriented link. Let $k \leq n$, $1 \leq l \leq n-k + 1$ be two fixed integers. We say that $\tau$ is \emph{$(k,l)$-admissible} if it admits an orientation where all  strands $l,l+1,\ldots,l+k-1$ have the same orientation (either all upwards, or all downwards). By extension if $D=\hat{\tau}$ we say that $D$ is $(k,l)$-admissible if and only if $\tau$ is $(k,l)$-admissible.
\end{definition}

Let $D$ be a diagram presented as the closure of a  $(n,n)$-tangle $\tau$. The $(k,l)$-limit of $D$, denoted by $\widetilde{C}(D,k,l)$, is the direct limit of the directed system below.
	\[
	\langle \widetilde{C}(D^q)[0,(k-1)(q(k-1)))] , f_{jq} \rangle, \mbox{ where } f_{jq} = i_{j} \circ i_{j-1} \circ \cdots \circ i_{q} \mbox{ if } q < j, \mbox{ and } f_{qq} = 1.
	\]
The corresponding homology will be denoted by $\widetilde{Kh}^{\ast,\ast}(D,k,l)$. Again, we have shifted the complexes so that the inclusions have bi-degree $(0,0)$ if and only if our tangle is $(k,l)$-admissible and equipped with the (one of the two) orientation that make it so. Any of these limit is invariant under planar isotopies and Reidemeister moves acting outside the region where we twist the strands. Note that for $(k,l)$-admissible tangles, we can add the twists one by one with consistent orientations at each step, even if the number of components of the closure might change. This nets us a directed system more in tune with the definition of $\widetilde{Kh}^{\ast,\ast}(T_{p,\infty})$, and both procedure yield the same result.

\medskip

Given a $(k,l)$-admissible tangle $\tau$, any $(k,l)$-limit $\widetilde{Kh}^{\ast,\ast}(\tau,k,l)$ can be made into a module over the ring $\widetilde{Kh}^{\ast,\ast}(T_{k,\infty})$. This structure is achieved through the use of a movie made of successive $1$-handle moves pictured below (for $k=3$, and $D$ the closure of a tangle $\tau$).

\begin{center}
	\includegraphics{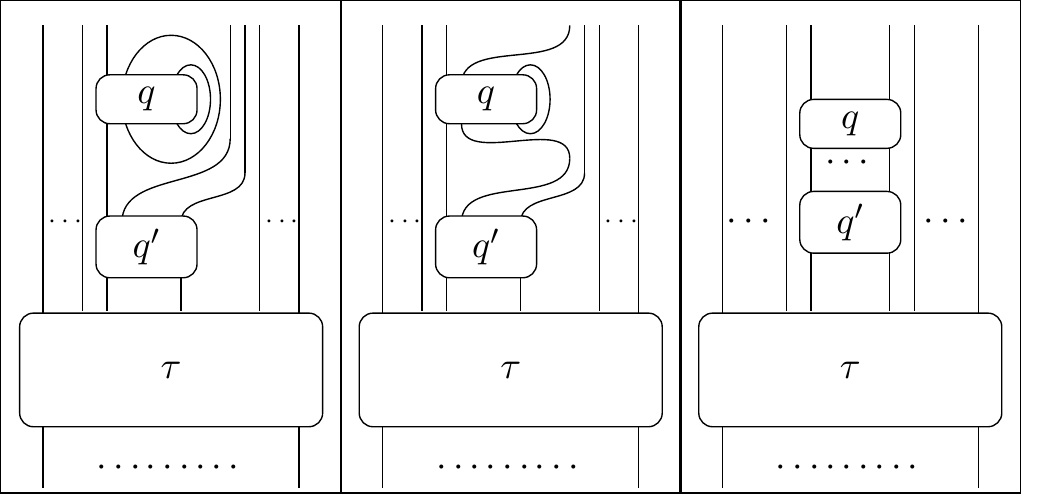}
\end{center}

The movie above produces an action of $\khpq{k}{\infty}$ on $\widetilde{Kh}^{\ast,\ast}(D,k,l)$, i.e a map:
\[
\Sigma_{(k,l)}:  \khpq{k}{\infty} \otimes \widetilde{Kh}^{\ast,\ast}(D,k,l)  \longrightarrow \widetilde{Kh}^{\ast,\ast}(D,k,l).
\]
All axioms of an action can be proved by adapting the proof of Theorem \ref{algstruc}. Thus $\widetilde{Kh}^{\ast,\ast}(D,k,l)$ can be considered a $\khpq{k}{\infty}$-module.

\begin{example}\label{unlink}
	Let $D=T_{p,1}$ and $\tau$ the underlying braid. For any  $2\leq k \leq n$, $1 \leq l \leq n-k + 1$ and any $q\geq 0$, we have $\hat{\tau}^q = T_{k,kq+1}$ so there is an isomorphism of $\widetilde{Kh}^{\ast,\ast}(T_{k,\infty})$-modules
	\[
	\widetilde{Kh}^{\ast,\ast}(D,k,l) \cong \widetilde{Kh}^{\ast,\ast}(T_{k,\infty}).
	\]
\end{example}

The total long exact sequence is compatible with this process. The limit functor is exact, so given a $(k,l)$-admissible diagram $D$ and for any choice of crossing, we get a sequence of ungraded vector spaces 
\[
0 \longrightarrow \widetilde{C}^{\ast}(D_1,k,l) \overset{i}{\longrightarrow} \widetilde{C}^{\ast}(D,k,l) \overset{q}{\longrightarrow} \widetilde{C}^{\ast}(D_0,k,l) \longrightarrow 0.
\]
In order to introduce gradings, we need to consider the shifts associated to our exact triple at each step. Let us make an easy observation: at least one of $D_0$ or $D_1$, the one with the induced orientation, is also $(k,l)$-admissible. But the other one might not be. This leads to the following consideration.

\medskip
\begin{lemma}\label{absorb}
	Let $D$ be a $(k,l)$-admissible diagram and $c$ be a crossing in $D$. 
	\begin{enumerate}[(i)]
		\item{If exactly one of $D_0$ or $D_1$ is $(k,l)$-admissible, then there is an isomorphism of $\khpq{k}{\infty}$-module 
			\[
			\widetilde{Kh}^{\ast,\ast}(D,k,l) \cong
			\begin{dcases*}
			\widetilde{Kh}^{\ast,\ast}(D_0,k,l)[0,1] & if $c$ is positive. \\
			\widetilde{Kh}^{\ast,\ast}(D_1,k,l)[0,-1] & if $c$ is negative.
			\end{dcases*}
			\]
		}
		\item{Let $\sigma_i$ be the $ith$ positive generator of the braid group over $n$ strands. If $l\leq i \leq l+k-2$, then  there are isomorphisms of $\khpq{k}{\infty}$-module 
			\[
			\widetilde{Kh}^{\ast,\ast}(\widehat{\sigma_i^{\pm 1}\circ \tau},k,l) \cong \widetilde{Kh}^{\ast,\ast}(\hat{ \tau},k,l)[0,\pm 1] \cong
			\widetilde{Kh}^{\ast,\ast}(\widehat{\tau \circ \sigma_i^{\pm 1}},k,l).
			\]
		}
		\item{If both are $(k,l)$-admissible, then for each $j \in \bbZ$ there is a limit short exact sequence:
			\[
			0 \longrightarrow \widetilde{C}^{\ast,j}(D_1,k,l)[w_{+},3w_{+}-1] \overset{i}{\longrightarrow} \widetilde{C}^{\ast,j}(D,k,l) \overset{q}{\longrightarrow} \widetilde{C}^{\ast,j}(D_0,k,l)[w_{-},3w_{-}+1] \longrightarrow 0,
			\]
			where $w_{+},w_{-}$ are the shifts associated to the original exact triple $(D_1,D,D_0)$. Additionally, the associated total long exact sequence respects the module structure.
		}
	\end{enumerate}
\end{lemma}

\begin{proof}
	We begin with a very simple observation concerning the negative crossings of $D^q$. Since $D$ is $(k,l)$-admissible, i.e. all twisting strands are oriented in the same direction, we must have
	\[
	n_{-}(D^q)= k(k-1) + n_{-}(D^{q-1}).
	\]
	For $(i)$, let us assume that the crossing $c$ is negative, so that $D_1$ is $(k,l)$-admissible and $D_0$ is not. Since $D_0$ is not $(k,l)$-admissible, neither are any of the $D^q_0$. In particular, since there are always at least $k$ positive crossings in each added full twist, we add less than $(k-1)(k-1)$ negative crossings each step of the way:
	\[
	n_{-}(D^q_0) \leq (k-1)(k-1) + n_{-}(D^{q-1}_0).
	\]
	Therefore the shift $w^{q}_{-}$ in the short exact sequence for $D^q$ is related to the previous one $w^{q-1}_{-}$:
	\[
	w^q_{-}= n_{-}(D^q_0) - n_{-}(D^q) \leq (k-1)^2 + n_{-}(D^{q-1}_0) -k(k-1) - n_{-}(D^{q-1}) = w^{q-1}_{-} -(k-1).
	\]
	The sequence $w^q_{-}$ is strictly decreasing as $k \geq 2$. Since $w_{+}=0$, we have the graded short exact sequence
	\[
	0 \longrightarrow \widetilde{C}^{\ast,\ast}(D_1,k,l)[0,-1] \overset{i}{\longrightarrow} \widetilde{C}^{\ast,\ast}(D,k,l) \overset{q}{\longrightarrow} \widetilde{C}^{\ast,\ast}(D_0,k,l)[-\infty,-\infty] \longrightarrow 0.
	\]
	All elements of $\widetilde{C}^{\ast,\ast}(D_0,k,l)[-\infty,-\infty]$ have arbitrarily small homological degree and so this complex is contractible. Taking homology yields the result. If $c$ positive, one should just reverse the roles of $D_0$ and $D_1$, thus exchanging the roles of $w_{-}$  and $w_{+}$.
	
	The second claim $(ii)$ is a straightforward application of $(i)$. Consider the exact sequence associated to the crossing $\sigma_i^{\pm 1}$. One of $D_1$ or $D_0$ cannot be $(k,l)$-admissible, because of the cup in the case $\sigma_i^{\pm 1}\circ \tau$ (or the cap for $\tau \circ \sigma_i^{\pm 1}$). Hence, by $(i)$, the claim holds.
	
	For $(iii)$, the process is roughly the same. Since both $D_1$ and $D_0$ are $(k,l)$-admissible, we can count precisely the number of negative crossings at each step. Since all $k$ strands are consistently oriented, we add only negative crossings for each new full twist. Thus, if $c$ is negative, for any $q\geq 1$ we have
	\[
	\begin{array}{ccl}
	w^q_{-} &=& n_{-}(D^q_0) - n_{-}(D^q)  \\
	&=& k(k-1) + n_{-}(D^{q-1}_0) - (k(k-1) + n_{-}(D^{q-1})) \\
	&=& n_{-}(D^{q-1}_0) - n_{-}(D^{q-1}) = w^{q-1}_{-}.
	\end{array}
	\]
	Therefore $w^q_{-}=w_{-}$ for any $q$. Since $w^q_{+}=0$, the graded short exact is given by 
	\[
	0 \longrightarrow \widetilde{C}^{\ast,j}(D_1,k)[0,-1] \overset{i}{\longrightarrow} \widetilde{C}^{\ast,j}(D,k) \overset{q}{\longrightarrow} \widetilde{C}^{\ast,j}(D_0,k)[w_{-},3w_{-}+1] \longrightarrow 0.
	\]
	The positive case is treated similarly. The total exact sequence preserves the module structure, courtesy of the monoidality and functoriality of both the $\underset{\rightarrow}{\lim}$ and Khovanov functors.
\end{proof}

\begin{remark}
	First let us say that both properties $(i)$ and $(ii)$ were already treated in Rozansky's work \cite{rozansky2010infinite}, albeit in a different setting. In the rest of this work, we will to refer to the isomorphisms in $(ii)$ as the \emph{absorption property}.
\end{remark}

Let us first give two interesting applications, that will come in handy later.

\begin{example}\label{shiftedinf}
	Let $\tau$ be the $3$-braid given by $\sigma_1^{-1} \circ \sigma_2^{-n}$, so that $\hat{\tau}=T_{2,n}$. Then there are isomorphisms
	\[
	\widetilde{Kh}^{\ast,\ast}(\hat{\tau},3,1) \cong \widetilde{Kh}^{\ast,\ast}(\widehat{\sigma_1^{-1} \circ \sigma_2^{-1}},3,1) [0,-n+1] \cong \khpq{3}{\infty}[0,-n+1],
	\] 
	where the first isomorphism is an iteration of Lemma \ref{absorb} $(ii)$, and the second uses the definition of $\khpq{3}{\infty}$.
	Also, let us consider the $3$-braid $\tau=\sigma_1^{-1}$, that closes to a $2$-component unlink. Then there is an isomorphism
	\[
	\widetilde{Kh}^{\ast,\ast}(\hat{\tau},3,1) \cong \widetilde{Kh}^{\ast,\ast}(\widehat{\sigma_1^{-1} \circ \sigma_2^{-1}},3,1)[0,1] \cong  \khpq{3}{\infty}[0,1].
	\]
	by inserting a crossing via Lemma \ref{absorb} again.
\end{example}

The functoriality of the limit, combined with $(iii)$ of Lemma \ref{absorb}, also nets us projections at the level of the total exact sequence. First, note that the canonical projection
\[
\varphi:\widetilde{C}^{\ast,\ast}(D) \longrightarrow \widetilde{C}^{\ast,\ast}(D,k,l)
\]
induces a map of short exact sequences:
\begin{center}
\[
\begin{array}{ccccccc}
0 \longrightarrow& \widetilde{C}^{\ast,\ast}(D_1)[w_{+},3w_{+}-1]& \overset{i}{\longrightarrow}& \widetilde{C}^{\ast,\ast}(D)& \overset{q}{\longrightarrow}& \widetilde{C}^{\ast,\ast}(D_0)[w_{-},3w_{-}+1]& \longrightarrow 0 \\
& \Big\downarrow \varphi^1&&\Big\downarrow \varphi&&\Big\downarrow \varphi^0&\\
0 \longrightarrow& \widetilde{C}^{\ast,\ast}(D_1,k,l)[w_{+},3w_{+}-1]& \overset{i}{\longrightarrow}& \widetilde{C}^{\ast,\ast}(D,k,l)& \overset{q}{\longrightarrow}& \widetilde{C}^{\ast,\ast}(D_0,k,l)[w_{-},3w_{-}+1]& \longrightarrow 0,
\end{array}
\]
\end{center}
where $\varphi^1:\widetilde{C}^{\ast,\ast}(D_1) \longrightarrow \widetilde{C}^{\ast,\ast}(D_1,k,l)$ and $\varphi^0:\widetilde{C}^{\ast,\ast}(D_0) \longrightarrow \widetilde{C}^{\ast,\ast}(D_0,k,l)$ are also projections. Thus at the level of the total exact sequence, we have a chain map $\varphi^0 \oplus \varphi^1$:
\[
\begin{array}{ccc}
\khr^{i,j}(D_0)[w_{-},3w_{-}+1]  & \bigoplus &  \khr^{i+1,j}(D_1)[w_{+},3w_{+}-1] \\
\Big\downarrow \varphi^0& & \Big\downarrow \varphi^1\\
\khr^{i,j}(D_0,k,l)[w_{-},3w_{-}+1]  & \bigoplus &  \khr^{i+1,j}(D_1,k,l)[w_{+},3w_{+}-1]\\
\end{array}
\]

With the absorption property, we can deepen our understanding of the general setting. We aim to show that as a $\khpq{p-1}{\infty}$-module, $\khpq{p}{\infty}$ is determined only by two families of total exact sequences.

\begin{center}
	\includegraphics[scale=0.75]{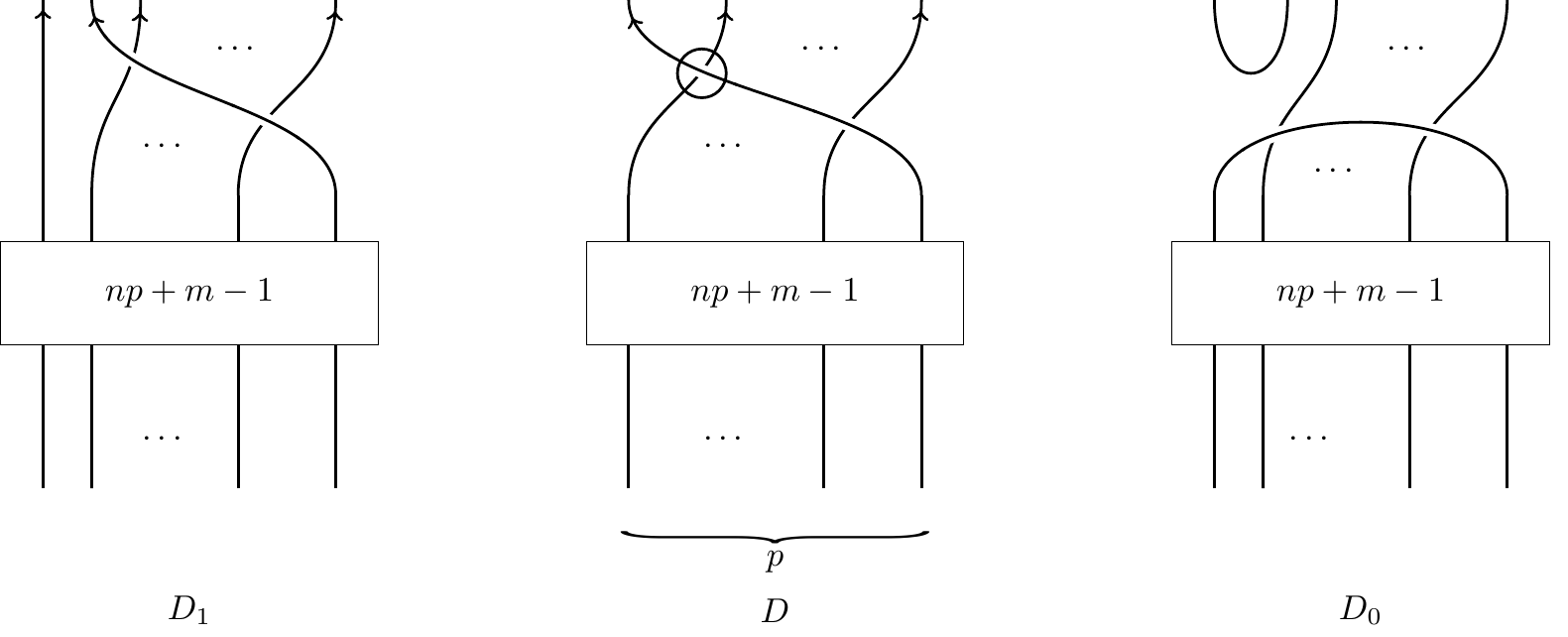}
		\captionof{figure}{The standard exact triple for torus links}\label{Generaltriple}
\end{center}

\begin{lemma}\label{intsteps}
	Fix $p\geq 3$. For any $2 \leq m \leq p-1$, and any positive integer $n \geq 0$, there are isomorphisms of graded $\khpq{p-1}{\infty}$-modules
	\[
	\khr^{\ast,\ast}(T_{p,np+m},p-1,2) \cong \khr^{\ast,\ast}(T_{p,np+1},p-1,2)[0,-(m-1)(p-1)].
	\]
\end{lemma}

\begin{proof}
	This statement follows readily from Lemma \ref{absorb} $(i),(ii)$. We consider the exact triple $(D_1,D,D_0)$ associated to the circled negative crossing in Figure \ref{Generaltriple}. For any $2 \leq m \leq p-1$, $D_0$ is not $(p-1,2)$-admissible, because of the cap, so we have
	\[
	\khr^{\ast,\ast}(T_{p,np+m},p-1,2) \cong \khr^{\ast,\ast}(D_1,p-1,2)[0,-1]
	\]
	We can now absorb the remaining $p-2$ negative crossings. This action induces an isomorphism
	\[
	\khr^{\ast,\ast}(D_1,p-1,2) \cong \khr^{\ast,\ast}(T_{p,np+m-1},p-1,2)[0,-p+2].
	\]
	The two previous isomorphisms can be combined into
	\[
	\khr^{\ast,\ast}(T_{p,np+m},p-1,2) \cong \khr^{\ast,\ast}(T_{p,np+m-1},p-1,2)[0,-p+1].
	\]
	If $m-1=1$, the result is as claimed. If $m-1 \geq 2$, we can repeat this process and each iteration nets us a $[0,-p+1]$ shift. This means that the inclusions induce isomorphisms
	\[
	\khr^{\ast,\ast}(T_{p,np+l},p-1,2) \cong \khr^{\ast,\ast}(T_{p,np+1},p-1,2)[0,-(l-1)(p-1)],
	\]
	as claimed.	
\end{proof}

Fortunately, this result does not hold if $l=p$ or $p+1$, since the resulting $D_0$ is $(p-1,2)$-admissible. For the next remark, let us omit the shifts associated to the long exact sequence. In both cases $l=p,p+1$, one checks easily that
\[
\khr^{\ast,\ast}(D_0,p-1,2) \cong \khpq{p-1}{\infty}.
\]
If $l=p+1$, our total exact sequence, which abuts to $\khr^{\ast,\ast}(T_{p,np+1},p-1,2)$, is just a family of boundary maps
\[
\khpq{p-1}{\infty} \longrightarrow \khr^{\ast,\ast}(T_{p,np},p-1,2),
\]
while if $l=p$, computing $\khr^{\ast,\ast}(T_{p,np},p-1,2)$ means studying a family 
\[
\khpq{p-1}{\infty} \longrightarrow \khr^{\ast,\ast}(T_{p,np-1},p-1,2) \cong \khr^{\ast,\ast}(T_{p,(n-1)p+1},p-1,2),
\]
where the last isomorphism was provided by Lemma \ref{intsteps}. To understand the algebra $\khpq{p}{\infty}$, one should first examine these two families of total exact sequences of $\khpq{p-1}{\infty}$-modules.

\section{The $3$-stranded torus links}

In this section, we begin by some computations. The two families whose homology we compute be central to our understanding of $\khpqd{3}{\infty}$ as a $\khpqd{2}{\infty}$-module. We then proceed to compute the algebra $\khpqd{3}{\infty}$.

\subsection{Examples of computations}

Recall that the limit total exact sequence is now a chain complex of $\khpq{k}{\infty}$-modules and that we can project the finite total exact sequences to the limit ones. This fact allows us to transfer information from finite cases to limits and vice-versa.
 
\begin{proposition}\label{332q}
 	The $\delta$-graded homology of $T_{(3,3),(2,q)}$ of Example \ref{T332q} is given by the following grid, where the $(-5,-q-2)$ is always empty:
 	
	\medskip
	\hspace*{1in}
 	\includegraphics{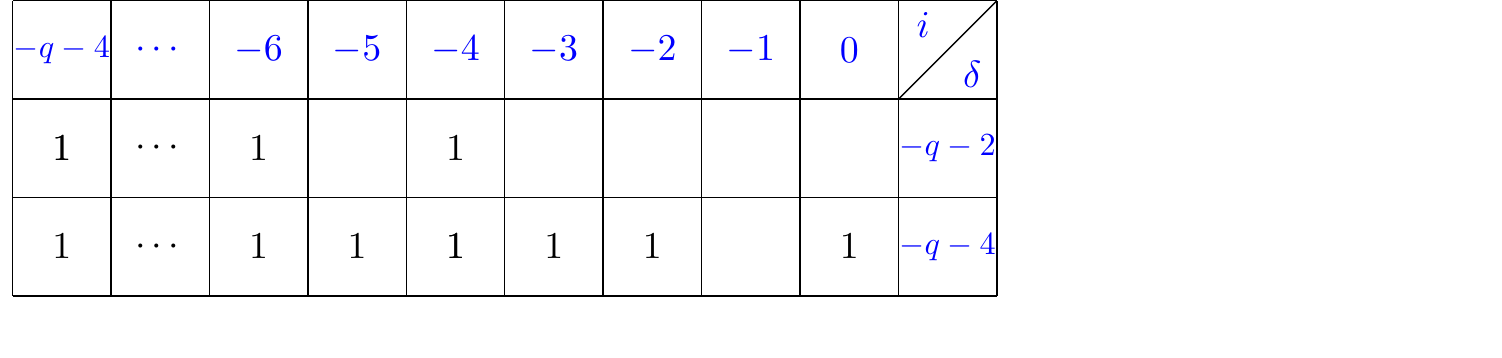}

\end{proposition}

\begin{proof}
	The proof is done in two steps, first we compute $\widetilde{Kh}^{\ast}_{\ast}(T_{3,3},2,2)$, then we use this information to understand all finite cases. For fixed $q$, the exact triple is given by $(T_{2,q},T_{(3,3),(2,q)},T_{2,q+4})$ and the corresponding shift is $w_{-}=-4$. Since all diagrams are $(2,2)$-admissible, we obtain the limit total sequence:
	
	\hspace*{-0.425in}
	\includegraphics[scale=0.75]{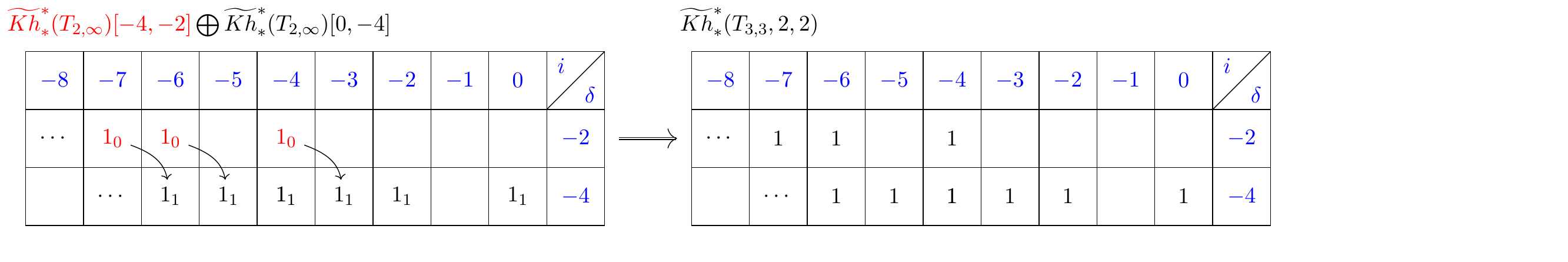}

	Let $z$ denote the generator of the $(-4,-2)$ entry of the left grid. This generator lies in the image of the canonical projection from $q=0$
	\[
	\varphi_0: \widetilde{Kh}^{\ast}_{\ast}(U \coprod U)[-4,-3] \longrightarrow \khr^{\ast}_{\ast}(T_{2,\infty})[-4,-2].
	\]
	Here there is a correction of shifts, due to the fact that the first diagram of our sequence is a $2$-component unlink (see Example \ref{shiftedinf}). Since the differential in the first stage (at $q=0$) is zero (see Example \ref{T332q}) then the limit differential from the entry $(-4,-2)$ to the entry $(-3,-4)$  must also be zero, i.e $d(z)=0$. Moreover, the total exact sequence is compatible with the module structure and $z$ generates a rank one free module, so the limit differential must be identically zero. It follows that $\widetilde{Kh}^{\ast}_{\ast}(T_{3,3},2,2)$ is given by the grid on the right. Let us now turn to the finite cases for $q \geq 1$. The corresponding pair of grids is given below.

	\hspace*{-0.425in}
	\includegraphics[scale=0.75]{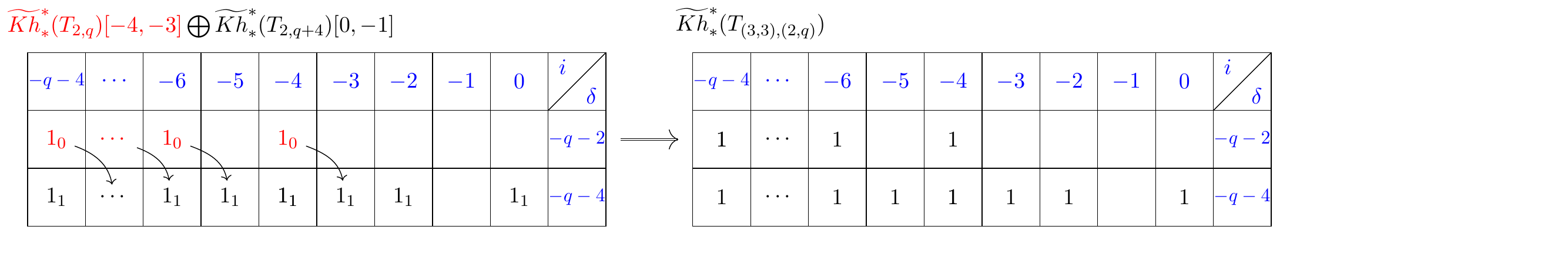}	
		
	One deduces the resulting grid by using the projections
	\[
	\varphi_q: \widetilde{Kh}^{\ast}_{\ast}(T_{2,q})[-4,-3] \longrightarrow \widetilde{Kh}^{\ast}_{\ast}(T_{2,\infty})[-4,-q - 2].
	\]
	Again, the shifts on the right hand side had to be corrected, since we start with $T_{2,q}$ and we have to absorb $q-1$ negative crossings.	These projections are injective: we can consider the finite total long exact sequence as a part of the limit one above. Thus it has the same differential, namely the zero map, and the result follows.
\end{proof}

The reader should be aware that the links $T_{(3,3),(2,q)}$ are isotopic to $3$-stranded pretzel links $P(-q-2,2,-2)$, whose (rational) homology has been computed by Manion \cite{manion2014rational}. Moreover note that along the way we have also shown that the projections
\[
	\varphi_q:\widetilde{Kh}^{\ast}_{\ast}(T_{(3,3),(2,q)}) \longrightarrow \widetilde{Kh}^{\ast}_{\ast}(T_{3,3},2,2)
\]
are all injective. We now turn to a similar family, based on the $T_{3,4}$ torus link. This family $T_{(3,4),(2,q)}$ is given by the middle diagram in the next figure.

\begin{center}
	\includegraphics[scale=0.6]{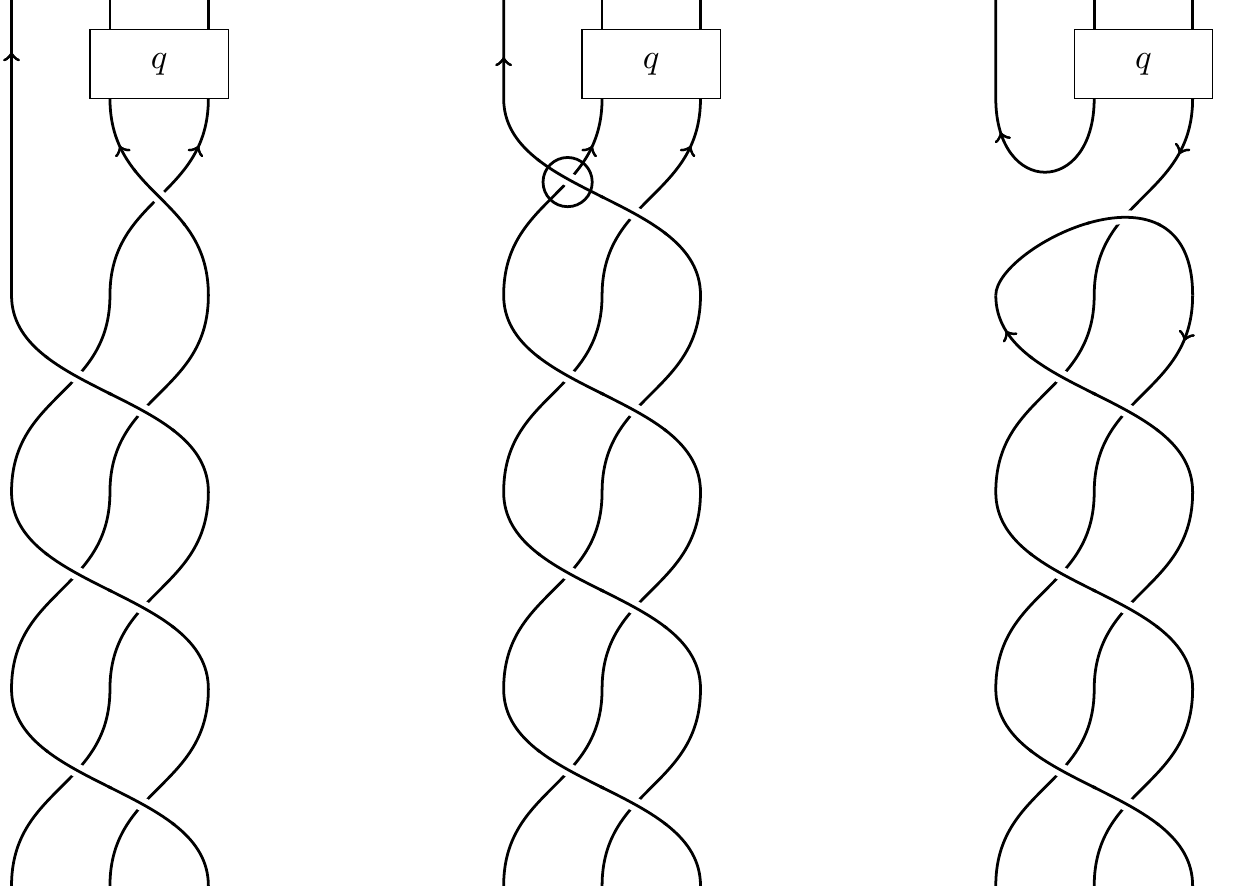}
	\captionof{figure}{The exact triple ($D_1,T_{(3,4),(2,q)},D_0)$}\label{T34triple}	
\end{center}

\begin{proposition}\label{342q}
	The $\delta$-graded homology of $T_{(3,4),(2,q)}$ is given by the grid below, where the $(-5,-q-4)$ slot is always empty:

	\medskip
	\hspace*{1in}
	\includegraphics{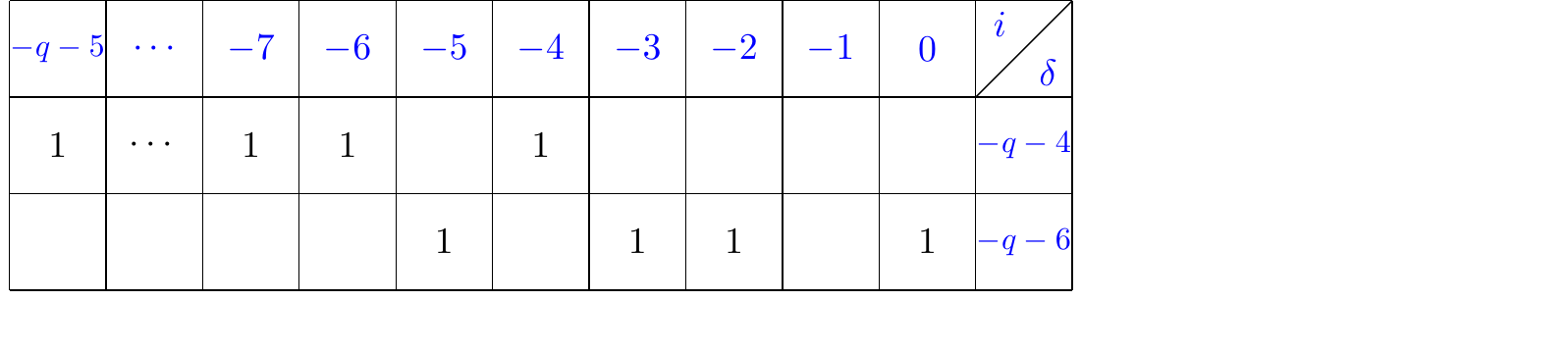}
	 
\end{proposition}

\begin{proof}
	The proof follows the same strategy as the one for Proposition \ref{332q}. We begin with the $T_{3,4}$ case, then compute a limit, from which we extract information about the finite cases. From Figure \ref{T34triple}, we find the exact triple $(T_{(3,3),(2,q+1)},T_{(3,4),(2,q)},T_{2,q+1})$. All $3$ diagrams are $(2,2)$-admissible so we will have a total exact sequence for any fixed $q$, and a limit one. The shift is given by $w_{-}= 3+q -(8+q)=-5$. We begin our study with case $q=0$. In this case we have the link $T_{(3,4),(2,0)}=T_{3,4}$ whose homology is known.
	
	\hspace*{0.3in}
	\includegraphics[scale=0.75]{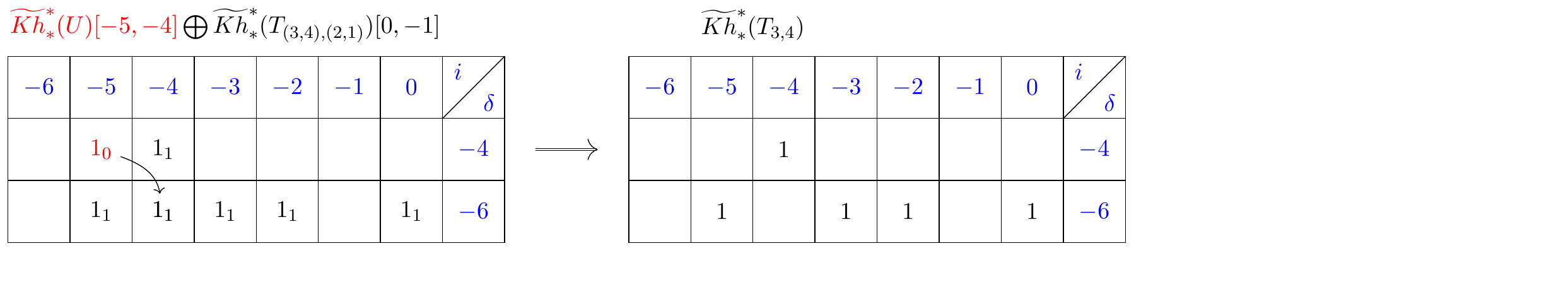}

	Comparing the data of the two grids above, one realises that the only possibly non-trivial differential (the arrow) is indeed non trivial, and even an isomorphism when restricted to the $(-5,-4)$ entry. For the limit case, we have the two grids below.

	\hspace*{-0.425in}
	\includegraphics[scale=0.75]{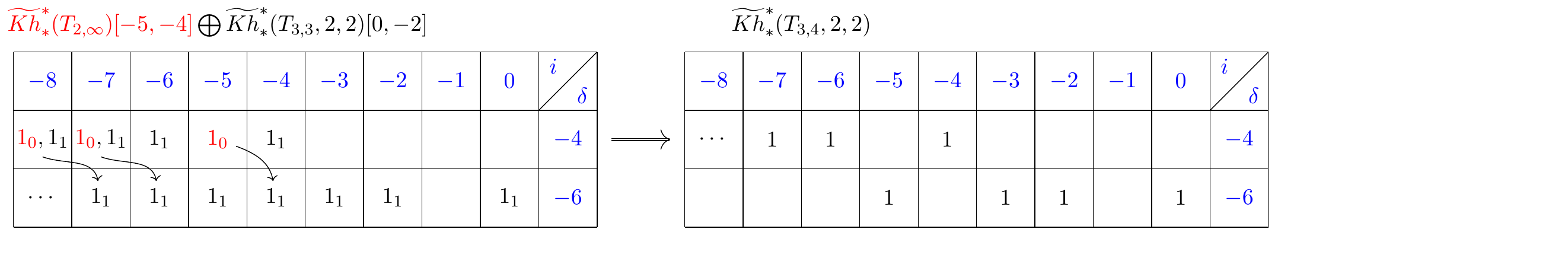}

	Let $v$ be the generator of the module $\widetilde{Kh}^{\ast}_{\ast}(T_{2,\infty})[-5,-4]$. Clearly $v$ lies in the image of the injective projection
	\[
	\varphi_q: \khr^{0}_{0}(U)[-5,-4] \longrightarrow \widetilde{Kh}^{0}_{0}(T_{2,\infty})[-5,-4].
	\] 
	Since the differential restricted to the $(-5,-4)$ entry in the first stage ($q=0$) is an isomorphism, and the projection is injective, then the restricted limit differential must also be non zero. For dimensional reasons, it must be an isomorphism (when restricted). Both $v$ and $d(v)$ generate a rank one free module, so these two modules must cancel each other. Thus the grid on the right. To extract information about the finite case, consider the pair of grids is given below.

	\hspace*{-1in}
	\includegraphics[scale=0.75]{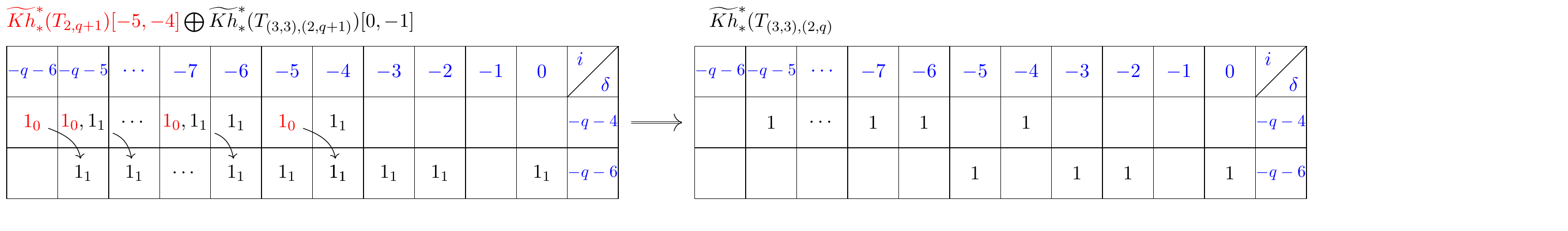}

	Since the projections are injective, we can again consider the total long exact sequence as a part of the limit one, which implies the differential restricts on each entry to a non-zero map. The result follows.
\end{proof}

As a final comment on these computations, let us mention that along the way, we have actually proved that there are isomorphism of $\khpqd{2}{\infty}$-modules
\[
\begin{array}{lcl}
\widetilde{Kh}^{\ast}_{\ast}(T_{3,3},2,2) & \cong & \widetilde{Kh}^{\ast}_{\ast}(T_{2,\infty})[0,-4] \bigoplus \widetilde{Kh}^{\ast}_{\ast}(T_{2,\infty})[-4,-2], \\
\khr^{\ast}_{\ast}(T_{3,3N+1},2,2) & \cong & \dfrac{\khpqd{2}{\infty}}{(a_2^2=a_3^2=0)}\left[0,-6 \right]  \bigoplus \widetilde{Kh}^{\ast}_{\ast}(T_{2,\infty})[-4,-4].
\end{array}
\]

\subsection{The algebra $\khpqd{3}{\infty}$}
We now turn our attention to the algebra structure for the $3$-stranded case. Again we will rely heavily on the fact that the homology of any $3$-stranded torus link is known. First we will study a particular $\khpqd{2}{\infty}$-module structure for $\khpqd{3}{\infty}$. From this point on we will be able to identify our algebra structure, thus proving the main theorem of this section. 
 
 \begin{proposition}
 	There is a bi-graded $\khpqd{2}{\infty}$-module isomorphism
 	\[
 	\khpqd{3}{\infty} \cong \bigoplus\limits_{i=0}^{\infty} \khpqd{2}{\infty}/(a_2^2=a_3^2=0)[-4i,2i],
 	\]
 	where $a_2$ and $a_3$ are the generators of the algebra $\khpq{2}{\infty}$ in homological degrees $-$2 and $-3$ respectively.
 \end{proposition}

 \begin{proof}
 	 The strategy for this proof is similar to the one used by Turner. We build the limit inductively, by showing that the two statements below are true.
 	\begin{enumerate}
 		\item{For any $N\geq 1$, there are graded $\khpqd{2}{\infty}$-module isomorphisms
 			\[
 			\begin{array}{ccl}
 			\khr^{\ast}_{\ast}(T_{3,3N},2,2) & \cong & \khpqd{2}{\infty}[-4N,-4N+2] \bigoplus \khpqd{2}{\infty}[-4(N-1),-4N]\\
 			&&\bigoplus_{0 \leq i \leq N-2} \dfrac{\khpqd{2}{\infty}}{(a_2^2=a_3^2=0)}\left[-4i,-6N+2+2i \right].
 			\end{array}
 			\]
 		}
 		\item{For any $N\geq 1$, there are graded $\khpqd{2}{\infty}$-module isomorphisms
 			\[
 			\begin{array}{ccl}
 			\khr^{\ast}_{\ast}(T_{3,3N+1},2,2) & \cong & \khpqd{2}{\infty}[-4N,-4N]\\
 			&&\bigoplus_{0 \leq i \leq N-1} \dfrac{\khpqd{2}{\infty}}{(a_2^2=a_3^2=0)}\left[-4i,-6N+2i \right],
 			\end{array}
 			\]
 		}
 		where $a_2$ and $a_3$ are the generators of $\khpqd{2}{\infty}$ in homological degree $-2$ and $-3$ respectively.	
 	\end{enumerate}
	The two cases for $N=1$ have already been treated in Propositions \ref{332q} and \ref{342q} respectively. We will show that if $(1)$ is true for $T_{3,3N}$ then $(2)$ holds for $T_{3,3N+1}$ and that if $(2)$ holds for $T_{3,3N+1}$, then $(1)$ holds for $T_{3,3(N+1)}$. Note that we can ignore the $T_{3,3N+2}$ torus knots, courtesy of Lemma \ref{intsteps}. Let us first compute the shifts associated to our long exact sequences for $T_{3,3N}$ and $T_{3,3N+1}$. With the orientations provided in Figure \ref{T33N}, we have
 	\[
 	w_{-}(T_{3,3N})=2N-6N=-4N, \hspace{1cm} w_{-}(T_{3,3N+1})=2N+1-(6N+2)=-4N-1.
 	\]
 	For any $N\geq 2$, the rest of the proof is identical to the base case, with arguments verbatim from Propositions \ref{332q} and \ref{342q}. We know from Turner's work that the total exact sequence for $T_{3,3N}$ using the top-most, left-most crossing has a trivial differential. This differential is transported to the limit exact sequence by the projections, thus the whole copy of $\khpqd{2}{\infty}[-4N,-4N+2]$ survives. Similarly, the differential in the total exact sequence for $T_{3,3N+1}$ is non-zero, and maps - at the limit - a generator of a copy of $\khpqd{2}{\infty}$ to an element that generates a rank one free $\khpqd{2}{\infty}$-module, namely $a_2^2$. Thus the whole copy disappears.

 	The description of the $\khpqd{2}{\infty}$-module structure is a straight-forward consequence of $(1)$ and $(2)$. When $q$ grows, we add a copy of $\khpqd{2}{\infty}$, which is then partially cancelled. This process repeats at each step. One should then proceed to shift this limit module so that the unique element in homological degree $0$ also has $\delta$-grading $0$. Since $\khr^{0}_{\ast}(T_{3,3N+1})$ is supported in $\delta$-grading $-6N$, we apply a $[0,6N]$ at each step. This concludes the proof.
 \end{proof}

 \begin{figure}
	\includegraphics[scale=0.75]{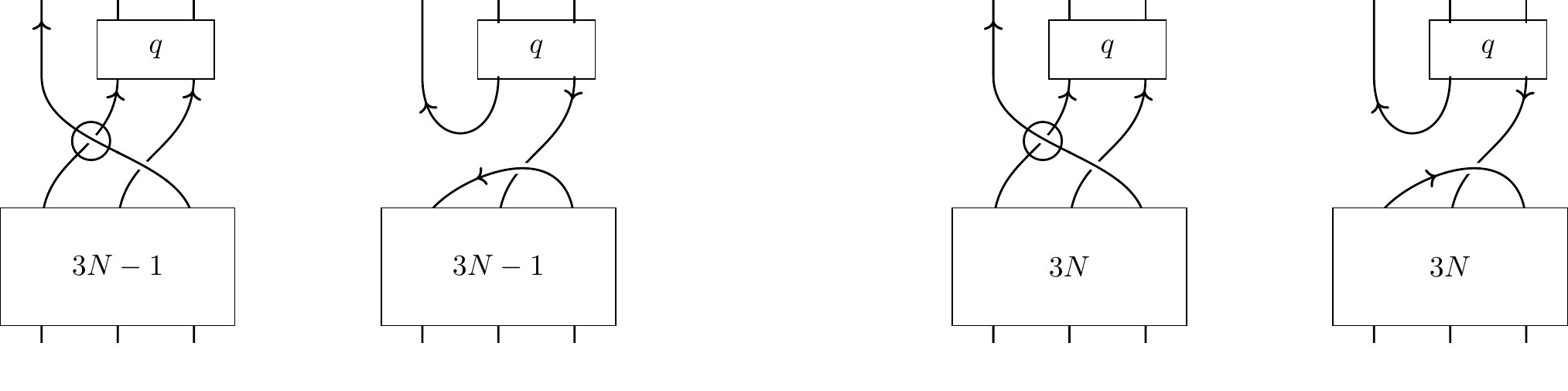}
	\captionof{figure}{On the left: the link $T_{(3,3N),(2,q)}$ and a $0$-smoothing companion. On the right: the link $T_{(3,3N+1),(2,q)}$ and a $0$-smoothing.}
	\label{T33N}
\end{figure}

 Before discussing the algebra, let us mention that one can also compute the homology of all intermediate diagrams, $T_{(3,3N),(2,q)}$ and $T_{(3,3N+1),(2,q)}$, by following the procedure of Propositions \ref{332q} and \ref{342q}.
 
 Thus $\khr^{\ast}_{\ast}(T_{3,\infty})$ should be understood as an infinite number of shifted copies of the row (which is a torsion module) of the form
 \[
 \begin{tabular}{c|c|c|c|c|c}
 &     &  &   & &     \\
 \hline  
 $1$& & $1$ & $1$ &      &  $1$ \\
 \hline
 &&&&&
 \end{tabular}
 \]
 In particular, the vector space $\khr^{-4N}_{2N}(T_{3,\infty})$ is one dimensional for all $N \geq 0$. From this point of view, it makes sense for us to look for a polynomial generator of $\khr^{\ast}_{\ast}(T_{3,\infty})$ - now seen as an algebra - that would take one row to the next one. This consideration leads to the main theorem of this section.
 
 \begin{theorem}\label{Algebra3}
 	There is a bi-graded algebra isomorphism
 	\[
 	\khr^{\ast}_{\ast}(T_{3,\infty}) \cong \bbZ_2[x,y,z]/(x^2=y^2=0),
 	\]
 	where $|x|=(-2,0),|y|=(-3,0)$ and $|z|=(-4,2)$.
 \end{theorem}
 
 \begin{proof}
 	As mentioned above, we only need to show that there is a polynomial generator  connecting two consecutive copies of $K:=\khpqd{2}{\infty}/(a_2^2=a_3^2=0)$. For $N\geq 1$, let $z_N$ be the generator of the $N$th copy $K[-4N,2N]$ of $K$ inside $\khr^{\ast}_{\ast}(T_{3,\infty})$. The element $z_N$ has bi-degree $(-4N,2N)$ and generated $\khr^{-4N}_{2N}(T_{3,\infty}) \cong \bbZ_2$. It is sufficient to show that the product structure
 	\[
 	\khr^{-4}_{2}(T_{3,\infty})\otimes \khr^{-4N}_{2N}(T_{3,\infty}) \longrightarrow \khr^{-4(N+1)}_{2(N+1)}(T_{3,\infty})
 	\]
 	maps $z_1 \otimes z_N$ to $z_{N+1}$. Indeed, this provides an identification $z_1 \cdot z_N = z_{N+1}$
 	and thus for all $N\geq 1$:
 	\[
 	z_1^N=z_N,
 	\]
 	i.e. $z_1$ is a polynomial generator. This is achieved by studying finite fusion maps and transport the result via the projections. Each $z_N$ appears first as an element of $\khr^{\ast}_{\ast}(T_{3,3N},2,2)$, so there must exist -abusing notations- $z_N \in \khr^{\ast}_{\ast}(T_{3,3N})$ that survives through the projection. Therefore we will consider the fusion maps
 	\[
 	\widetilde{\Sigma}^3_{3,3N}: \khr^{\ast}_{\ast}(T_{3,3}) \otimes \khr^{\ast}_{\ast}(T_{3,3N}) \longrightarrow \khr^{\ast}_{\ast}(T_{3,3(N+1)})  
 	\]
 	For each $N\geq 2$, consider the total long exact sequence for $T_{3,3N}$ with respect to the leftmost topmost crossing. The associated exact triple is given by $(D^N_0,D^N,D^N_1)= (U \coprod U, T_{3,3N}, D^N_1)$. Here we omit a description of the diagrams $D^N_1$, since they are irrelevant to the proof. Fusion maps induce maps of short exact sequences, so for $D_0^N$ we have:
 	\[
 	\Sigma_N: \khr^{\ast}_{\ast}(T_{3,3}) \otimes \khr^{\ast}_{\ast}(U \coprod U)[-4N,-4N+2] \longrightarrow \khr^{\ast}_{\ast}(U \coprod U)[-4(N+1),-4(N+1)+2].
 	\]
 	This map is surjective by lemma \ref{Alg3}. Additionally we have
 	\[
 	| z_1 \otimes z_N| = (-4,-2) + (-4N,-4N+2) = (-4(N+1),-4N).
 	\]
 	Since the fusion map is a succession of two $1$-handles it has a bi-degree $(0,-2)$. This means that $|\Sigma_N(z_1 \otimes z_N)| = (-4(N+1),-4(N+1)+2)$. The restriction of $\Sigma_N$ to these bi-degree is a surjective mapping between one dimensional vector spaces therefore an isomorphism. Thus our original map (for the $D^N$'s), namely
 	\[
 	\Sigma_N:\khr_{-2}^{-4}(T_{3,3}) \otimes \khr^{-4N}_{-4N+2}(T_{3,3N}) \longrightarrow \khr^{4(N+1)}_{-4(N+1)+2}(T_{3,3(N+1)})
 	\]
 	is also an isomorphism at these degrees. All three vector spaces are preserved by the projection to the limit, hence we have an isomorphism (after re-grading)
 	\[
 	\Sigma_{\infty}^{\ast}:\khr_{2}^{-4}(T_{3,\infty}) \otimes \khr^{-4N}_{2N}(T_{3,\infty}) \longrightarrow \khr^{-4(N+1)}_{2(N+1)}(T_{3,\infty})
 	\]
 	which gives the desired identification $z_1 \cdot z_N = z_{N+1}$ and concludes the proof.
 \end{proof}

\section{The $4$-stranded torus links}
 We move on to the case of $4$-stranded torus links. We shall compute the algebra $\khpqd{4}{\infty}$ and use it to obtain lower bounds for the width of $\khpqd{4}{q}$, for any $q >0$.

\subsection{The algebra structure}

The $2$ and $3$ stranded cases were treated by relying upon explicit descriptions of $\khpqd{2}{q}$ and $\khpqd{3}{q}$ for all $q \geq 0$. However the reduced Khovanov homology of $4$-stranded torus links is still unknown in general. It follows, that in order to understand $\khpqd{4}{\infty}$, we must become as independent as possible from the explicit description of  $\khpqd{4}{q}$.

Let us begin with the family of $T_{4,4N+1}$ torus knots. The exact triple $(D^N_1,D^N,D^N_0)$ below yields a total sequence that abuts to $\khr^{\ast}_{\ast}(T_{4,4N+1})$. 
 \begin{center}
 	\includegraphics[scale=0.75]{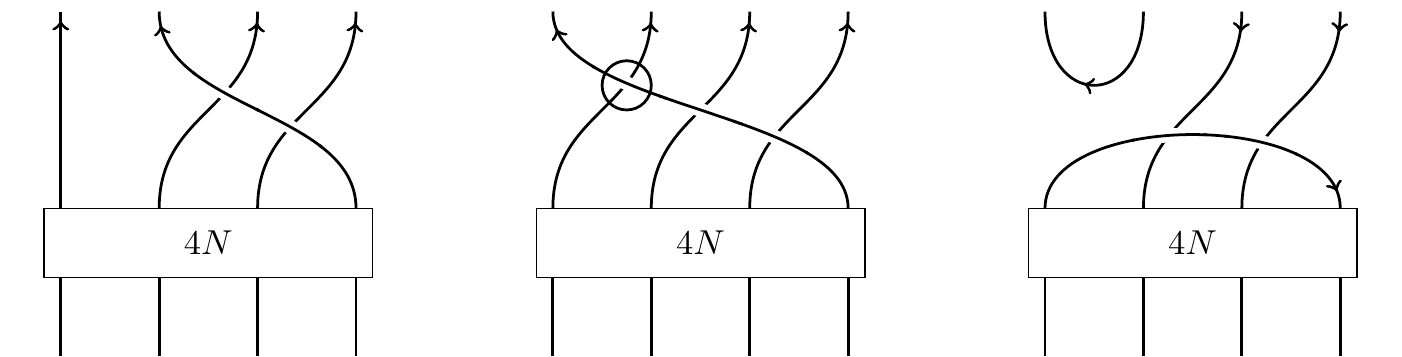}
 	\captionof{figure}{The exact triple for $T_{4,4N+1}$}\label{T44triple}
 \end{center}
 Clearly $D^N_0=T_{2,2N+1}$.
 The shifts for all $N \geq 1$ can be computed using the orientations provided, and:
 \[
 w^N_{-}= -6N-1.
 \]
We start with the case $N=1$. Using the KnotTheory Mathematica package, we obtained a description of both $\khr^{\ast}_{\ast}(T_{4,5})$ and $\khr^{\ast}_{\ast}(D^1_1)$. Our present goal is to understand the differential. The associated grid is given by:
 \begin{center}
 	\includegraphics[scale=0.75]{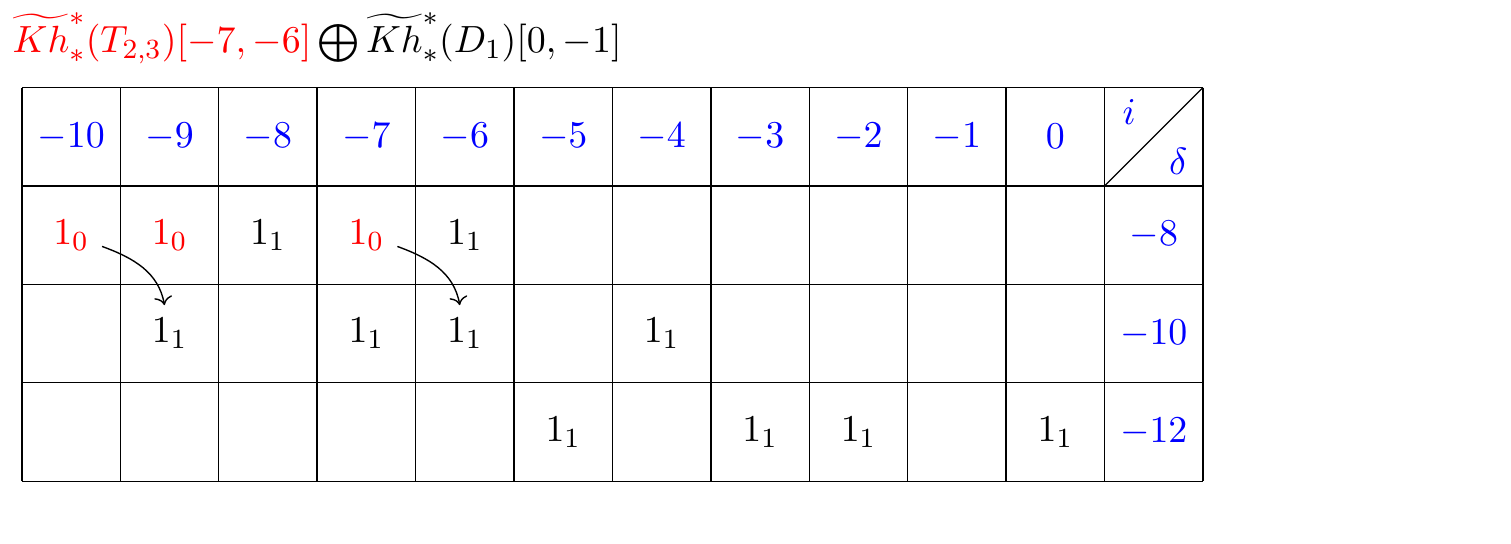}
 \end{center}
 And it converges to:
 \begin{center}
 	\includegraphics[scale=0.75]{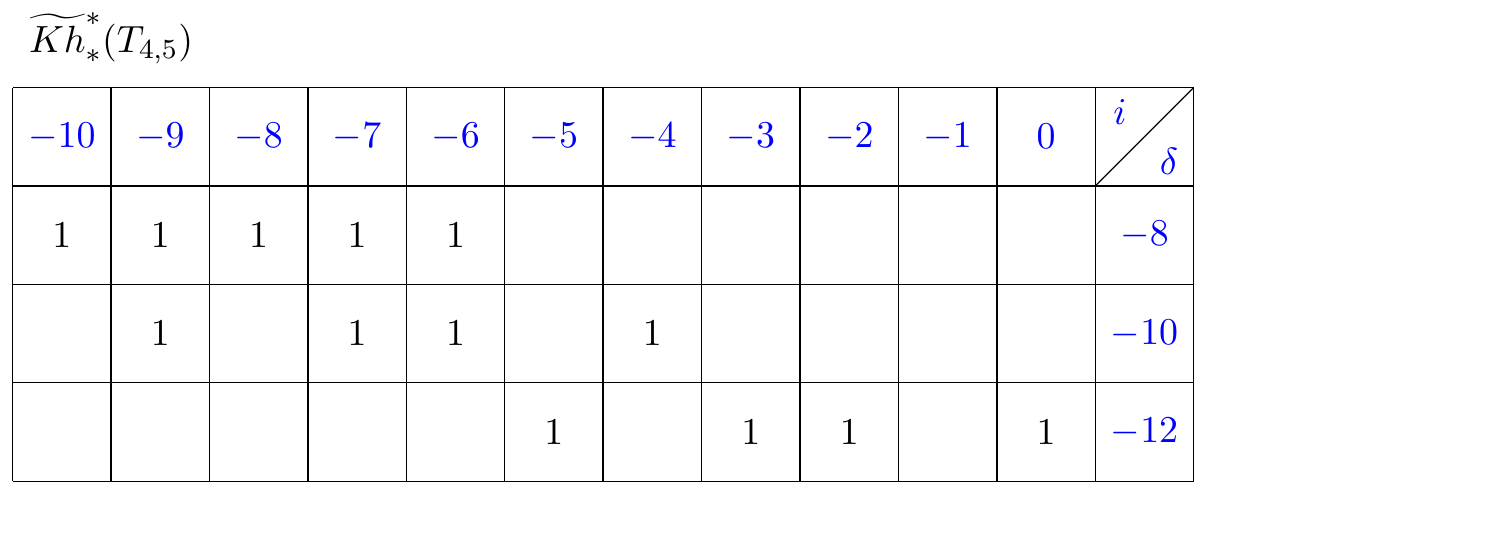}
 \end{center}
 The differential is thus the zero map.
 In particular, the element $w \in \khr^{0}_{-2}(D_0=T_{2,3})$, at the $(-7,-8)$ slot of the grid, survives.
 
 We now turn to the module $\khr^{\ast}_{\ast}(T_{4,5},2,3)$. The total sequence has underlying vector space
 $\khr^{\ast}_{\ast}(T_{3,\infty})[-7,-8] \bigoplus \khr^{\ast}_{\ast}(D_1,2,3)$.
 The canonical projection maps $w$ to the generator of the shifted rank one free module. Since $d(w)=0$, the whole module survives: 
 \[
 d(\khr^{\ast}_{\ast}(T_{3,\infty}))=0.
 \]
 Thus an isomorphism 
 \[
 \khr^{\ast}_{\ast}(T_{4,5},2,3) \cong \khr^{\ast}_{\ast}(T_{3,\infty})[-7,-8] \bigoplus\khr^{\ast}_{\ast}(T_{4,4},2,3)[0,-3].
 \]
 A complete description of $\khr^{\ast}_{\ast}(T_{4,4})$ is given in Section $5$, but for now we need to know that  $\khr^{-6}_{-5}(T_{4,4})\cong \bbZ_{2}$. Consider Figure \ref{T44triple}, with $4N-1$ blocs instead of $4N$ blocs. With our usual choice of crossings, we obtain the exact triple $(D_1^N, T_{4,4N}, D_0^N = U \coprod T_{2,2N})$. All diagrams are $(3,2)$-admissible, with associated shift $w_{-}=-6N$.
We now have everything we need to state, then prove, the main result of this section.

\begin{theorem}\label{Alg4}
	There is a bi-graded algebra isomorphism
	\[
	\widetilde{Kh}^{\ast}_{\ast}(T_{4,\infty}) \cong \dfrac{\mathbb{Z}_2[x,y,z,v,w]}{(x^2=y^2=0, w^2 = \alpha vz^2 + \beta xv^2 )},
	\]
	where $\alpha,\beta \in \bbZ_2$. The degrees are given by
	\[
	|x|=(-2,0), \hspace{0.1cm}|y|=(-3,0), \hspace{0.1cm} |z|=(-4,2), \hspace{0.1cm} |v|=(-6,4), \hspace{0.1cm} |w|=(-7,4).
	\]
\end{theorem}

\begin{proof}
Let us start with a warning: the proof contains a lot of compounded shifts. We strongly advise the reader to go through Lemma \ref{intsteps} and Example \ref{shiftedinf} again, as they contains the properties (derived from Lemma \ref{absorb}) we will refer to. We follow the strategy suggested by the discussion that followed Lemma \ref{absorb}, i.e.  we need only understand two families of total exact sequences. We begin with the family for the $4N+1$ case, with our usual choice of crossing (see Figure \ref{T44triple}) and $(D_1^N,D^N=T_{4,4N+1},D_0^N=T_{2,2N+1})$ be the associated triple. 
\[
\khr^{\ast}_{\ast}(D_0^N)[-6N,-6N] \cong \khpqd{3}{\infty}[-6N-1,-6N -2N]  \overset{\partial}{\longrightarrow} \khr^{\ast}_{\ast}(T_{4,4(N-1)},3,2)[0,-3].
\]
Here the additional $-2N$ shift is due to the fact that $\khr^{0}_{\delta}(D_0^N) \neq 0$ if and only $\delta = -2N$ (more details in Example \ref{shiftedinf}). Our first step is to obtain isomorphisms
\[
\khr^{\ast}_{\ast}(T_{4,4N+1},3,2) \cong \khr^{\ast}_{\ast}(T_{3,\infty})[-6N-1,-8N] \bigoplus\khr^{\ast}_{\ast}(T_{4,4(N-1)},3,2)[0,-3].
\]
The main argument relies on fusion maps of the form
\[
\widetilde{\Sigma}^4_{4,4N+1}:  \khpqd{4}{4N+1}  \otimes \khpqd{4}{4} \longrightarrow \khpqd{4}{4(N+1)+1}.
\]
We know, from Lemma \ref{4surj}, that for any $N\geq 1$ they induce a surjective map
\[
\Sigma_{4N+1}:\khr^{\ast}_{\ast}(D_0^N)[-6N-1,-6N] \otimes   \khpqd{4}{4} \longrightarrow \khr^{\ast}_{\ast}(D_0^{N+1})[-6(N+1)-1,-6(N+1)].
\]
So it restricts to an isomorphism (between one dimensional vector spaces) for a particular choice of degrees, where the shifts have been taken into account:
\[
\Sigma_{4N+1}:\khr^{0}_{-2N}(D_0^N)\otimes \khr^{-6}_{-5}(T_{4,4}) \longrightarrow \khr^{0}_{-2(N+1)}(D_0^{N+1}).
\]
For the same choice of degrees, the projections are non-zero and yield a commutative diagram:
\[
\begin{array}{ccc}
\bbZ_2 \cong \khr^{0}_{-2N}(D_0^N)\otimes  \khr^{-6}_{-5}(T_{4,4}) &\overset{\Sigma_{4N+1}}{\longrightarrow}& \khr^{0}_{-2(N+1)}(D_0^{N+1}) \cong \bbZ_2 \\
\Big\downarrow \varphi \otimes 1 & & \Big\downarrow \varphi \\
\bbZ_2 \cong	\khr^{0}_{0}(T_{3,\infty})[0,-2N]\bigotimes \khr^{-6}_{-5}(T_{4,4}) & \overset{\Sigma^{\infty}_{4N+1}}{\longrightarrow} &	\khr^{0}_{0}(T_{3,\infty})[0,-2(N+1)] \cong \bbZ_2. \\
\end{array}
\]
Hence $\Sigma^{\infty}_{4N+1}$ is an isomorphism. Since $\khr^{\ast}_{\ast}(T_{3,\infty})$ is a rank one free module generated by the element in bi-degree $(0,0)$, this isomorphism extends to an isomorphism of modules.  We can now consider the limit total exact sequence, via the commutative diagram below, where vertical maps are fusions maps (shifts omitted)
\[
 	\begin{array}{ccc}
	\khr^{\ast}_{\ast}(T_{3,\infty}) \bigotimes \khr^{-6}_{-5}(T_{4,4}) & \overset{d^N \otimes 1}{\longrightarrow} & \khr(D_1^N,3,2) \bigotimes \khr^{-6}_{-5}(T_{4,4})\\
 	 \Sigma^{\infty}_{4N+1}  \big \downarrow& & \big \downarrow \\
 	\khr^{\ast}_{\ast}(T_{3,\infty}) & \overset{d^{N+1}}{\longrightarrow} & \khr(D_1^{N+1},3,2).
 	\end{array}
\]
We know from our previous study of $\khr^{\ast}_{\ast}(T_{4,5},2,2)$ that $d^1=0$. As shown above, the left-hand side vertical map is an isomorphism for all $N\geq 1 $. By induction starting at $N=1$, all differentials $d^N$ must vanish. Hence the isomorphism:
\[
\khr^{\ast}_{\ast}(T_{4,4N+1},3,2) \cong \khr^{\ast}_{\ast}(T_{3,\infty})[-6N-1,-8N] \bigoplus\khr^{\ast}_{\ast}(T_{4,4(N-1)},3,2)[0,-3].
\]
Let $w_N$ be the generator of the copy $\khpqd{3}{\infty}[-6N+1,-6N]$ and $v$ that of $\khr^{-6}_{-5}(T_{4,4})$. The earlier isomorphism of modules translates into an identification, for $N\geq 2$:
\[
w_N = w_{N-1}\cdot v.
\]
In particular, we have shown that $w_{N+1}=w_1v^N = (\cdots(((w_1v)v)v\cdots) \neq 0 \in \khpqd{4}{\infty}$. Thus, by associativity, $v^N \neq 0$ in $\khpqd{4}{\infty}$. The latter information can be used to understand the remaining family of total exact sequences, i.e.
\[
\khpqd{3}{\infty}[-6N,-8N]  \longrightarrow \khr^{\ast}_{\ast}(T_{4,4(N-1)+3},3,2)[0,-3] \cong \khr^{\ast}_{\ast}(T_{4,4(N-1)+1},3,2)[0,-9].
\]
The isomorphism is a particular case of Lemma \ref{intsteps}. Let $v_N$ be the generator of the $N^{th}$ copy $\khpqd{3}{\infty}[-6N,-8N]$. Since for all $N$, $w_1v^{N-1}$ supports a rank one free $\khpqd{3}{\infty}$-module, it follows that $v^N$ itself supports such a module. This module must be the one above, as there is no other choice. Thus $v_{N} = v^N$ and so $d^N(v_N)=0$ for any $N \geq 1$. This translates into an isomorphism
\[
\khr^{\ast}_{\ast}(T_{4,4N},3,2) \cong \khr^{\ast}_{\ast}(T_{3,\infty})[-6N,-8N] \bigoplus \khr^{\ast}_{\ast}(T_{4,4(N-1)+1},3,2)[0,-9].
\]
We now know the behaviour of our $\khr^{\ast}_{\ast}(T_{3,\infty})$ when we add a full twist:
\[
\begin{array}{ccl}
\khr^{\ast}_{\ast}(T_{4,4N+1},3,2)& \cong& \khr^{\ast}_{\ast}(T_{3,\infty})[-6N-1,-8N] \bigoplus \khr^{\ast}_{\ast}(T_{3,\infty})[-6N,-8N] \\
&&\bigoplus \khr^{\ast}_{\ast}(T_{4,4(N-1)+1},3,2)[0,-12]
\end{array}
\]
Let us summarize were we stand. Each time we add a full twist over $4$ strands, we add two copies of $\khr^{\ast}_{\ast}(T_{3,\infty})$, shifted suitably and supported by $v^N$ and $w_1v^N$ respectively. We will give an explicit description of $\khpqd{4}{\infty}$ as a $\khr^{\ast}_{\ast}(T_{3,\infty})$-module in a moment. Before that we need to normalize the directed system. The homology of $T_{4,4N+1}$ in homological degree $0$ is supported in $\delta$-grading $-12N$, so we apply a $[0,12N]$ shift. Hence for each $N\geq 1$, we have an isomorphism (of $\khr^{\ast}_{\ast}(T_{3,\infty})$-modules), after shift:
\[
\begin{array}{ccl}
\khr^{\ast}_{\ast}(T_{4,4N+1},3,2)& \cong& \khr^{\ast}_{\ast}(T_{3,\infty})[-6N-1,4N] \bigoplus \khr^{\ast}_{\ast}(T_{3,\infty})[-6N,4N] \\
&&\bigoplus \khr^{\ast}_{\ast}(T_{4,4(N-1)+1},3,2)[0,12(N-1)].
\end{array}
\]
Thus the limit gives us the final $\khr^{\ast}_{\ast}(T_{3,\infty})$-module:
\[
\khpqd{4}{\infty} \cong \khr^{\ast}_{\ast}(T_{3,\infty}) \oplus  \bigoplus\limits_{N=1}^{\infty} 
\left( \khr^{\ast}_{\ast}(T_{3,\infty})[-6N-1,4N] \oplus \khr^{\ast}_{\ast}(T_{3,\infty})[-6N,4N]
\right).
\]
One of these families is generated by elements in even homological degrees $-6N$, the other in odd homological degrees $-6N-1$. Moreover, the fusion map tell us we can jump from one copy to the next by multiplication by $v \in \khr^{-6}_{4}(T_{4,\infty})$. So we have a polynomial generator of our algebra, $v$. At this point it is quite clear that there is no room for $w_1$ to also be a polynomial generator. Indeed if that were the case, there would be more copies of $\khr^{\ast}_{\ast}(T_{3,\infty})$ living inside $\khr^{\ast}_{\ast}(T_{4,\infty})$. The element $w_1^2$ has a degree
\[
|w_1^2| = 2|w_1| = 2(-7,4) = (-14,8).
\]
One checks easily that there are only two other elements of the algebra $\khpqd{4}{\infty}$ with such a degree, namely $vz^2$ and $v^2x$. These three must then be linearly dependent. The obvious identification yields an isomorphism:
\[
\widetilde{Kh}^{\ast}_{\ast}(T_{4,\infty}) \cong \dfrac{\mathbb{Z}_2[x,y,z,v,w]}{(x^2=y^2=0, w^2 = \alpha vz^2 + \beta xv^2 )},
\]
where $\alpha,\beta \in \bbZ_2$. This concludes the proof.
\end{proof}
 
\subsection{The width of $\khpqd{4}{q}$} As an application of the description of the algebra structure, we give a lower bound to the width the family of $4$-stranded torus links.

\begin{corollary}\label{widthT4}
	The following inequalities hold
	\begin{enumerate}
		\item {$2n + 1 \leq w(\khpqd{4}{4n})$,\hspace{0.1cm}$w(\khpqd{4}{4n+1})$,\hspace{0.1cm}$w(\khpqd{4}{4n+2})$.
			}
		\item {$2n + 2 \leq w(\khpqd{4}{4n+3})$.
			}
	\end{enumerate}
\end{corollary}

Note that $2n + 1 \leq w(\khpqd{4}{4n})$ was already known from Sto\v{s}i\'c's work \cite{stovsic2007homological}.
\medskip

\begin{proof}
	For $(1)$, the result follows from the fact that $v^n$ can be understood as an element of $\khpqd{4}{4n}$, $\khpqd{4}{4n+1}$, $\khpqd{4}{4n+2}$ via the canonical projections. It appears first as an element of $\khr^{\ast}_{\ast}(T_{3,3N},3,2)$ and survives all the way to $\khpqd{4}{\infty})$. As an element of  $\khpqd{4}{\infty})$, it has bi-degree $(-6n,4n)$. Therefore, since the shifts do not alter the width, it follows that for $q=4n,4n+1,4n+2$:
	\[
	w(\khpqd{4}{q}) \geq \frac{1}{2}(4n - 0) + 1 = 2n + 1.
	\]
	For $(2)$, we consider not $v^n$, but $v^nz$. We have the commutative diagram, where the vertical maps are fusion maps and the horizontal maps are projections
	\[
	\begin{array}{ccc}
	\khpqd{4}{4n} \bigotimes \khpqd{4}{3} & \longrightarrow & \khpqd{4}{\infty} \bigotimes \khpqd{4}{\infty} \\
	\big \downarrow  & & \big \downarrow \\
	\khpqd{4}{4n+3} & \longrightarrow & \khpqd{4}{\infty}.
	\end{array}
	\]
	Since $v^n \in \khpqd{4}{4n}$, $z \in \khpqd{4}{3}$,  $v^nz \neq 0 \in \khpqd{4}{\infty}$, we must have that $v^nz \neq 0$ in $\khpqd{4}{4n+3}$ also. As an element of $\khpqd{4}{\infty}$, $v^nz$ has bi-degree $(-6n + 2, 4n +2)$. It follows that
	\[
	w(\widetilde{Kh}(T_{4,4n+3})) \geq \frac{1}{2}(4n + 2 - 0)  + 1 = 2n+2.
	\]
\end{proof}

\subsection{Comparison with the work of Gorsky-Oblomkov-Rasmussen}
We conclude this section by comparing our result with the Gorsky-Oblomkov-Rasmussen conjecture \cite{gorsky2013stable} (whose statement we adapt to our conventions). Consider the polynomial ring in variables $x_1, \ldots x_{p-1}$ and an equal number of odd variables $\chi_1,\ldots, \chi_{p-1}$ bi-graded as
\[
|x_k|= (-2k, 2k-2), \hspace{1cm} |\chi_k|=(-2k-1,2k-2).
\]
The differential $d$ is given by the formula
\[
d(x_k)=0, \hspace{1cm} d(\chi_k)= \sum_{i=1}^{k-1}x_ix_{k-i}.
\]

\begin{conj}[Gorsky-Oblomkov-Rasmussen,\cite{gorsky2013stable}]
The stable homology $\khpqd{p}{\infty}$ is isomorphic to the homology of $\bbZ_2[x_1,\ldots,x_{p-1}]\otimes \Lambda[\chi_1,\ldots,\chi_{p-1}]$ with respect to $d$.
\end{conj}

One checks easily that the expected stable homology $H_c(T_{p,\infty})$ can be described by
\[
\begin{array}{ccl}
H_c(T_{2,\infty}) &=& \bbZ_2[x_1,\chi_1]/(\chi_1^2=0), \\
H_c(T_{3,\infty}) &=& \bbZ_2[x_1,x_2,\chi_1]/(x_1^2=\chi_1^2=0), \\
H_c(T_{4,\infty}) &=& \bbZ_2[x_1,x_2,x_3,\chi_1,\chi_3]/(x_1^2=\chi_1^2=\chi_3^2=0). \\
\end{array}
\]

The reader should be aware that there exists a corresponding statement for the triply-graded homology that generalizes the HOMFLY-PT polynomial. Moreover, that version has been shown to be true by Hogoncamp \cite{hogancamp2015stable}, for integer coefficients. Let us also recall that there is a spectral sequence, due to Rasmussen \cite{rasmussen2016some}, that starts at the triply-graded homology and converges to the usual Khovanov homology. Hence it seemed reasonable to think that the conjecture also held for the usual variant. However, for $p=2$, our result is slightly different. This mismatch is actually a common defect of spectral sequences, these compute associated graded objects, rather than a homology itself. Consider the filtration on $\khpqd{p}{\infty}$ defined as
\[
0 \subset \calF^0(\khpqd{p}{\infty}) = \{ u \in \khpqd{p}{\infty} | u \mbox { has even homological degree } \} \subset \khpqd{p}{\infty}.
\]
Clearly $\calF^0(\khpqd{p}{\infty})$ is sub-algebra, so we have a filtered algebra and the associated graded algebra is given by
\[
\calG_p = \calF^0(\khpqd{p}{\infty}) \bigoplus \dfrac{\khpqd{p}{\infty}}{\calF^0(\khpqd{p}{\infty})}.
\]
The second summand contains the elements with odd homological degree and their squares must be zero. By definition of  $\calG_p$, there is an isomorphism of vector spaces $\calG_p \cong \khpqd{p}{\infty}$, and even an isomorphism of algebras for $p=3$. For the two other cases $p=2$ and $p=4$, it also follows immediately that there is an algebra isomorphism
\[
\calG_p \cong H_c(T_{p,\infty}).
\]
Hence it is $\calG_p$ and not $subset \khpqd{p}{\infty}$ that satisfies the conjecture.
 
\section{Odds and ends: surjective maps}
This section is dedicated to a method we use in order to study surjectivity of maps induced by $1$-handle moves. The key technique is a way to fit these maps into long exact sequences, which we call \emph{completing the triple}.

Consider an oriented $1$-handle move between two diagrams $F$ and $F'$.

\begin{center}
	\includegraphics{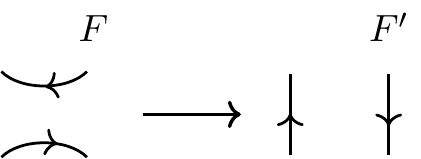}
\end{center}

If the two strands shown in $F$ belong to different components, then reverse the orientation of one of them to obtain a diagram $\bar{F}$.
\begin{center}
	\includegraphics{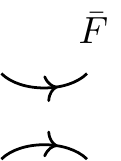}
\end{center}
Let $D$ be the diagram identical to $\bar{F}$ except in the region shown where it is:
\begin{center}
	\includegraphics{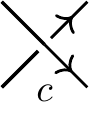}
\end{center}
$D$ has one more crossing, call it $c$, than $\bar{F}$. This crossing is positive and $D_0 = \bar{F}$. Moreover the two strands appearing in $D_1$ must be in the same component and we may choose the orientation to be that of $F'$.

If the two strands shown in $F$ belong to the same component, then the two strands shown in $F'$ must belong to different components. Thus one can reverse the orientation of any of these, yielding a diagram $\bar{F}'$.
\begin{center}
	\includegraphics{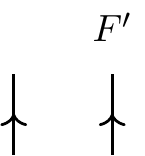}
\end{center}
Let $D$ be the diagram identical to $\bar{F}'$ except in the region shown, where it is
\begin{center}
	\includegraphics{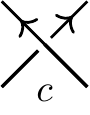}
\end{center}
$D$ has one more crossing than $\bar{F}'$. This crossing is negative and $D_1=\bar{F}'$. The orientation of $D_0$ can be chosen to be the orientation of $F$. In both case we obtain an exact triple $(D_1, D, D_0)$. We have summarized this procedure in Figure \ref{orientationchange}: on the left is the negative case, on the right the positive one.

\begin{center}
	\includegraphics{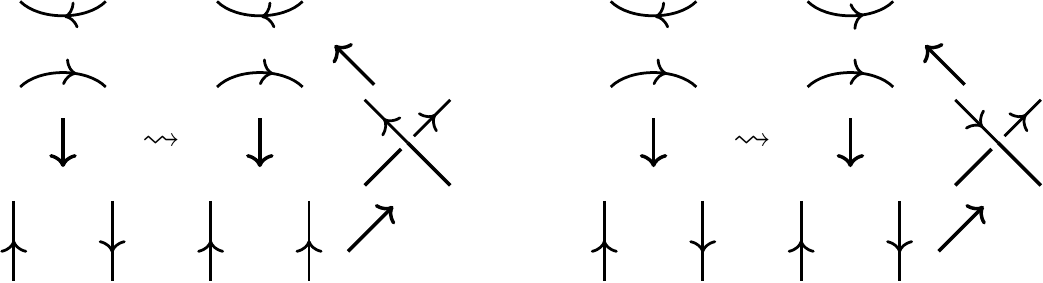}
	\captionof{figure}{The two ways of completing a triple.}\label{orientationchange}
\end{center}

The main ingredient of the procedure is the reversal of the orientation of a component, therefore we must know how the reduced Khovanov homology behaves with respect to this operation.

\begin{proposition}\label{proporchange}
	Let $D$ be an pointed diagram. If $D^r$ is obtained from $L$ by reversing the orientation of the $r$th component, then there are isomorphisms
	\[
	\begin{array}{l}
	\khr^{i,j}(D^r) \cong \khr^{i+2l_r, j+6l_r}(D), \\
	\\
	\khr^{i}_{\delta}(D^r) \cong \khr^{i+2l_r}_{\delta +2l_r}(D),
	\end{array}
	\]
	where $l_r = \sum\limits_{i \neq r} lk(L_r, L_i)$ and $L_i$ is the $i^{th}$ component of $D$.
\end{proposition} 

The isomorphism in question is in fact just the identity. However, reversing an orientation changes the crossings and thus the degrees. This means the identity has (a priori) a non trivial bi-degree. It is a straightforward computation from the definitions of the homological and quantum degrees to show that $|id|=(2l_r,6l_r)$. The $\delta$-graded version follows since $\delta=j-2i$.

\begin{lemma}\label{homalg}
	Let $E,F,G$ be a diagram such that 
	\[
	\khr^{\ast,\ast}(E) \cong \khr^{\ast,\ast}(F)[a_F,b_F] \oplus \khr^{\ast,\ast}(G)[a_G,b_G].
	\]
	If a diagram $D$ admits a crossing $c$ whose associated exact triple is either $(F,E,G)$ or $(G,E,F)$, then the boundary map in the long exact sequence is surjective.
\end{lemma}

\begin{proof}
	The proof is standard homological algebra.
\end{proof}

\begin{lemma}\label{triplethin}
	Let $(D_1,D,D_0)$ be an exact triple of thin diagrams, and let $[a_1,b_1],[a_0,b_0]$ be the associated shifts. Then the following holds:
	\begin{enumerate}[(i)]
		\item{If $w(\khr^{\ast}_{\ast}(D_1)[a_1,b_1]\oplus \khr^{\ast}_{\ast}(D_0)[a_0,b_0]) = 1$, then the boundary map is zero.}
		\item{If $w(\khr^{\ast}_{\ast}(D_1)[a_1,b_1]\oplus \khr^{\ast}_{\ast}(D_0)[a_0,b_0]) = 2$ and  $\dim (\khr^{\ast}_{\ast}(D_1)) \geq \dim(\khr^{\ast}_{\ast}(D_0))$, then the boundary map is injective.}
		\item{If $w(\khr^{\ast}_{\ast}(D_1)[a_1,b_1]\oplus \khr^{\ast}_{\ast}(D_0)[a_0,b_0]) = 2$ and  $\dim (\khr^{\ast}_{\ast}(D_1)) \leq \dim(\khr^{\ast}_{\ast}(D_0))$, then the boundary map is surjective.}
	\end{enumerate}
\end{lemma}

\begin{proof}
	For $(i)$, recall that the differential has degree $(1,-2)$. Under the thinness assumption, the differential must vanish. For $(ii)$ and $(iii)$, we start with the following observations: if $w(\khr^{\ast}_{\ast}(D_1)[a_1,b_1]\oplus \khr^{\ast}_{\ast}(D_0)[a_0,b_0]) = 2$, with $D_1$ and $D_0$ thin, then the two homologies must be supported in two different diagonals. Since $D$ is also thin, one the two diagonals must vanish. It follows that the boundary map is either injective or surjective. The hypotheses relating the dimensions of $\khr^{\ast}_{\ast}(D_1)$ and $\khr^{\ast}_{\ast}(D_0)$ allow us to decide when it is either.
\end{proof}

\begin{center}
	\includegraphics[scale=0.65]{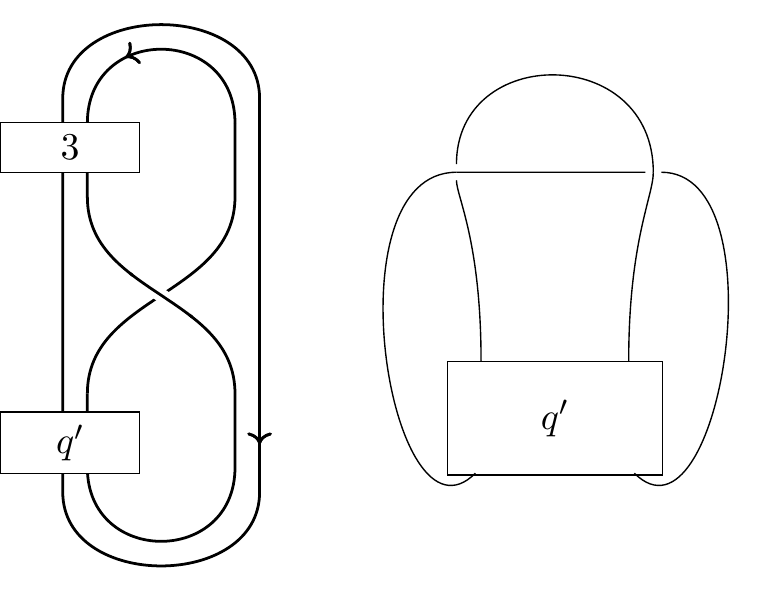}
	\captionof{figure}{Two diagrams for the alternating knot $L_{q'}$ ($q'$ odd) in the proof of lemma \ref{2fusion}.}\label{Twist}
\end{center}

\begin{lem}
	For any $q \geq 3$, the fusion maps
	\[
	\begin{aligned}
	\widetilde{\Sigma}_{2,q}^2: \widetilde{Kh}^{\ast}_{\ast}(T_{2,2}) \otimes \widetilde{Kh}^{\ast}_{\ast}(T_{2,q}) \longrightarrow \widetilde{Kh}^{\ast}_{\ast}(T_{2,2+q}), \\
	\widetilde{\Sigma}_{3,q}^2: \widetilde{Kh}^{\ast}_{\ast}(T_{2,3}) \otimes \widetilde{Kh}^{\ast}_{\ast}(T_{2,q}) \longrightarrow \widetilde{Kh}^{\ast}_{\ast}(T_{2,3+q})
	\end{aligned}
	\]
	are surjective.
\end{lem}

\begin{proof}
Let us treat the two cases separately. The maps in both cases are induced by oriented $1$-handle moves. Therefore we shall first "complete the triples".
For the first map, the movie begins with $F = T_{2,2}\sharp T_{2,q'}$ and $F' = T_{2,q'+2}$. The two strands of $F$ involved in the procedure don't belong to the same component, so we modify $F$. The triple is completed by a diagram $D$ equivalent to a negative torus link $T_{2,q'-2}$. The dimensions if these homologies are related by the formula:
\[
\dim(\khr^{\ast}_{\ast}(T_{2,2}\sharp T_{2,q'})) = 2q' = (q'-2) + (q'+2) =\dim{\khr^{\ast}_{\ast}(T_{2,q'-2})} + \dim(\khr^{\ast}_{\ast}(T_{2,q'+2})).
\]
Thus, for dimensional reasons the map $\Sigma^2_{2,q'}$ must be surjective.

The second map is induced by a movie starting at $F = T_{2,3}\sharp T_{2,q'}$ and ending at $F' = T_{2,q'+3}$. If $q'$ is even, then the strands in $F$ involved in the procedure belong to different components, so we modify $F$. If $q'$ is odd, one should modify $F'$ instead. In both cases, the triple is completed by a diagram $D$ equivalent to the knot $L_{q'-2}$ of Figure \ref{Twist}. All diagrams involved in the exact triple are alternating thus thin by Lee \cite{lee2008khovanov}. So from Lemma \ref{triplethin}, the boundary map is either $0$, surjective or injective. First notice that it cannot be injective since, for $q' \geq 2$:
\[
\dim(\khr^{\ast}_{\ast}(T_{2,3}\sharp T_{2,q'})) = 3q' > q' + 3 =\dim(\khr^{\ast}_{\ast}(T_{2,q'+3}))
\]
It remains to show that the boundary map cannot be zero. First we bound the dimension of the homology of $L_q$. If $q=0$, $L_0$ is an unknot $U$. Assume $q \geq 1$, consider the exact triple $(D_1, D, D_0)$ associated to one of the $q$ crossings. The diagram $D_0$ is equivalent to a $T_{2,2}$ torus link, whose homology has dimension $2$. For any $q \geq 1$, we have the inequality
\[
\dim(\khr^{\ast}_{\ast}(L_q)) \leq 2 + \dim(\khr^{\ast}_{\ast}(L_{q-1})),
\]
from which one extracts
\[
\dim(\khr^{\ast}_{\ast}(L_q)) \leq 2(q-1) + 1 = 2q - 1
\]
We shall prove that the boundary is non zero by contradiction. Let us assume that it is. Then we have
\[
4q' + 3 =\dim(\khr^{\ast}_{\ast}(T_{2,3}\sharp T_{2,q'})) + \dim(\khr^{\ast}_{\ast}(T_{2,q'+3})) = \dim(\khr^{\ast}_{\ast}(L_{q'-2})) \leq 2(q'-2) -1.
\]
The assumption gives the second equality. The inequality does not hold, the map $\Sigma^2_{3,q'}$ is non zero. It must then be surjective.
\end{proof}

\begin{lemma}\label{Alg3}
	For all $N\geq 1$, the map induced by the $3$-frame movie in Figure \ref{3movie}
	\[
	\Phi_N: \khr_{\ast}^{\ast}(E) \otimes \khr_{\ast}^{\ast}(U \coprod U) \longrightarrow \khr_{\ast}^{\ast}(U \coprod U),
	\]where $E$ is a $T_{3,3}$ with a change of orientation, is surjective.
\end{lemma}

\begin{figure}
	\hspace*{-.125in}
	\includegraphics[scale=0.65]{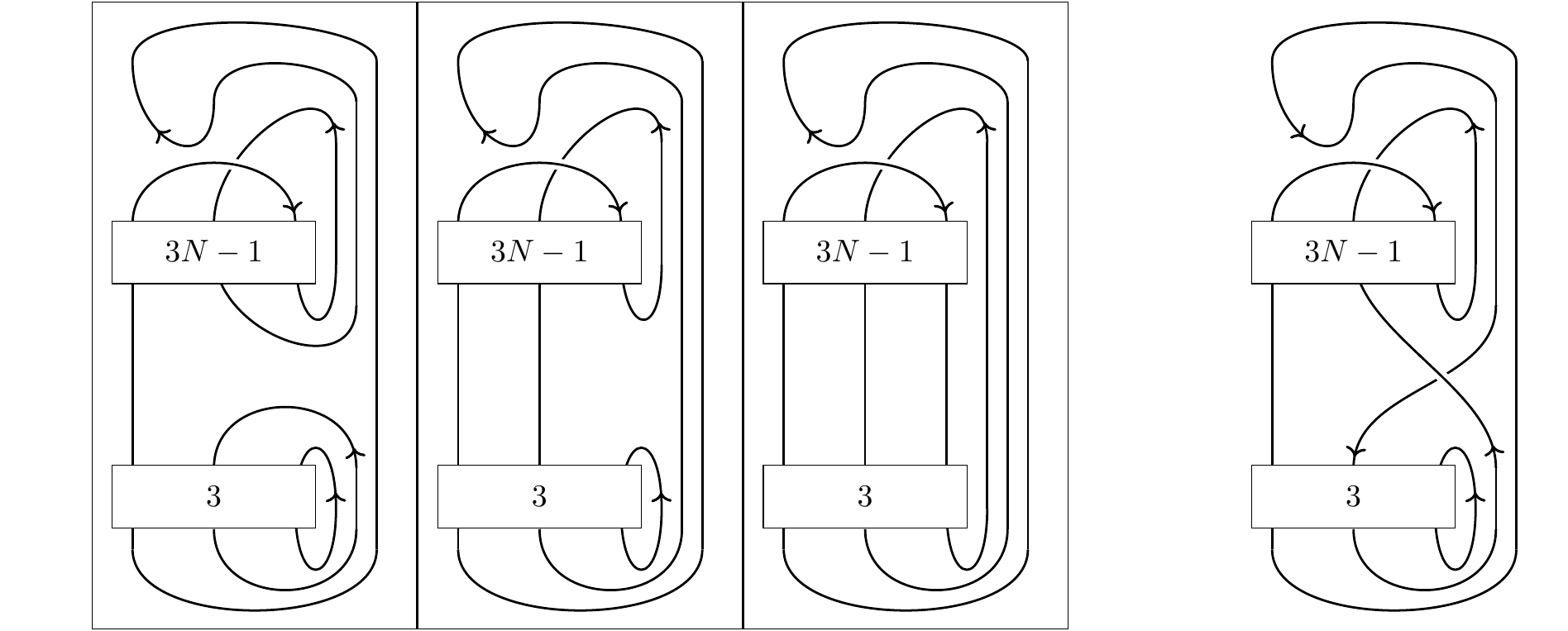}
	\captionof{figure}{On the left: the fusion movie for Lemma \ref{Alg3}. On the right: the completed link associated to the first $1$-handle move.}
	\label{3movie}
\end{figure}

\begin{proof}
	The diagrams involved in the movies can be easily described for any $N$:
	\[
	M_1^N = E \coprod U, \hspace{0,5cm} M_2^N = U\coprod U \coprod U, \hspace{0,5cm} M_3^N = U \coprod U.
	\]
	For any $N \geq 1$, the map $\Phi_N$ is defined as the composite:
	\[
	\Phi_N:\khr_{\ast}^{\ast}(E \sharp (U \coprod U)) \overset{\Phi_1^N}{\longrightarrow} \khr_{\ast-1}^{\ast}(U\coprod U \coprod U) \overset{\Phi_2^N}{\longrightarrow} \khr_{\ast-2}^{\ast}(U\coprod U).
	\]
	To prove that $\Phi_N$ is surjective, it is enough to show that both maps $\Phi_1^N$ and $\Phi^N_2$ are surjective. The map $\Phi_2^N$ is surjective as it realizes a connected sum. We now show that $\Phi_1^N$ is surjective. We  use our completing the triple procedure. First step is to reverse the orientation of one strand of $E$, turning it into a $T_{3,3}$. The complete link is easily shown to be isotopic to a $T_{2,4} \coprod U$ for any $N \geq 1$. Thus the completed exact triple for the $1$-handle move we consider is the same for all $N$, namely: 
	\[
	(M_2^N= U \coprod U \coprod U=  (U\coprod U) \sharp (U\coprod U) , T_{2,4}\sharp (U\coprod U), T_{3,3}\sharp (U\coprod U) ).
	\]
	Recall from Example \ref{T332q} that we have an isomorphism
	\[
	\begin{array}{lcl}
	\khr^{\ast,\ast}(T_{3,3}) &\cong& \khr^{\ast,\ast}(U \coprod U)[-4,-11] \bigoplus \khr^{\ast,\ast}(T_{2,4})[0,-1], \\
	\end{array}
	\]
	from which we can deduce the sequence of isomorphisms
	\[
	\begin{array}{lcl}
	\khr^{\ast,\ast}((U\coprod U)\sharp T_{3,3}) & \cong & \khr^{\ast,\ast}((U\coprod U))\otimes \khr^{\ast,\ast}(T_{3,3}) \\
	&\cong & \khr^{\ast,\ast}((U\coprod U))\otimes \left(\khr^{\ast,\ast}(U \coprod U)[-4,-11] \bigoplus \khr^{\ast,\ast}(T_{2,4})[0,-1]\right) \\
	&\cong & \left(\khr^{\ast,\ast}((U\coprod U))\otimes\khr^{\ast,\ast}(U \coprod U)[-4,-11]\right) \\
	& &\bigoplus \left(\khr^{\ast,\ast}((U\coprod U))\otimes \khr^{\ast,\ast}(T_{2,4})[0,-1]\right) \\
	&\cong & \khr^{\ast,\ast}((U\coprod U)\sharp(U \coprod U))[-4,-11] \bigoplus \khr^{\ast,\ast}((U\coprod U)\sharp T_{2,4})[0,-1].
	\end{array}
	\]
	This isomorphism, our exact triple and Lemma \ref{homalg}  allow us to conclude that $\varphi_1^N$ is indeed surjective. This concludes the proof.
\end{proof}

In order to prove our last statement, we need to understand the total exact sequence for $T_{4,4}$, with respect to our usual choice of topmost, leftmost crossing. The associated exact triple is
$(D_1=F,T_{4,4},U \coprod T_{2,2})$. Hence the grids

	\hspace*{-1.95in}
	\includegraphics[scale=0.75]{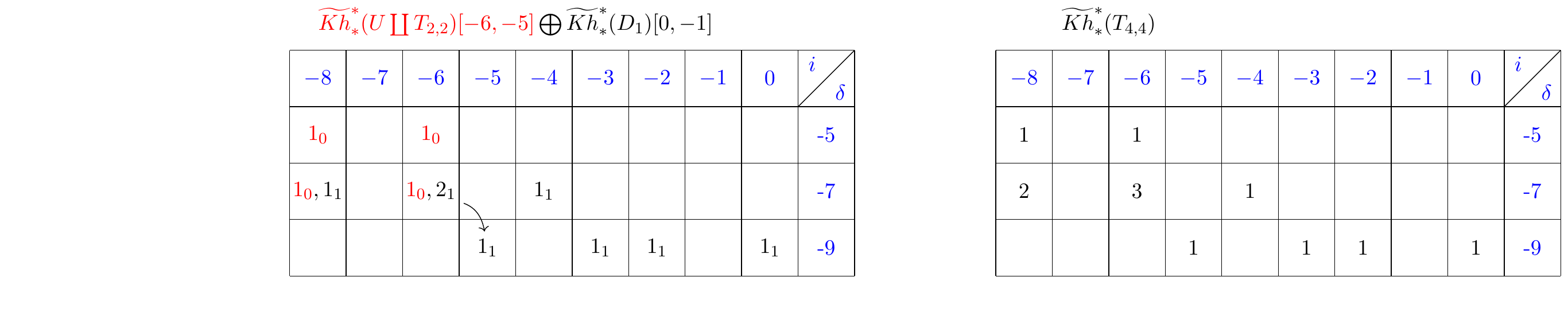}

Comparing the two grids, it is clear that the differential is identically zero, so we have an isomorphism
	\[
	\khr^{\ast}_{\ast}(T_{4,4}) \cong \khr^{\ast}_{\ast}((U\coprod T_{2,2})) \bigoplus \khr^{\ast}_{\ast}(F).
	\]

Our last order of business is to study the map induced by the movie below.

\begin{figure}
	\hspace*{-.5in}
	\includegraphics[scale=0.65]{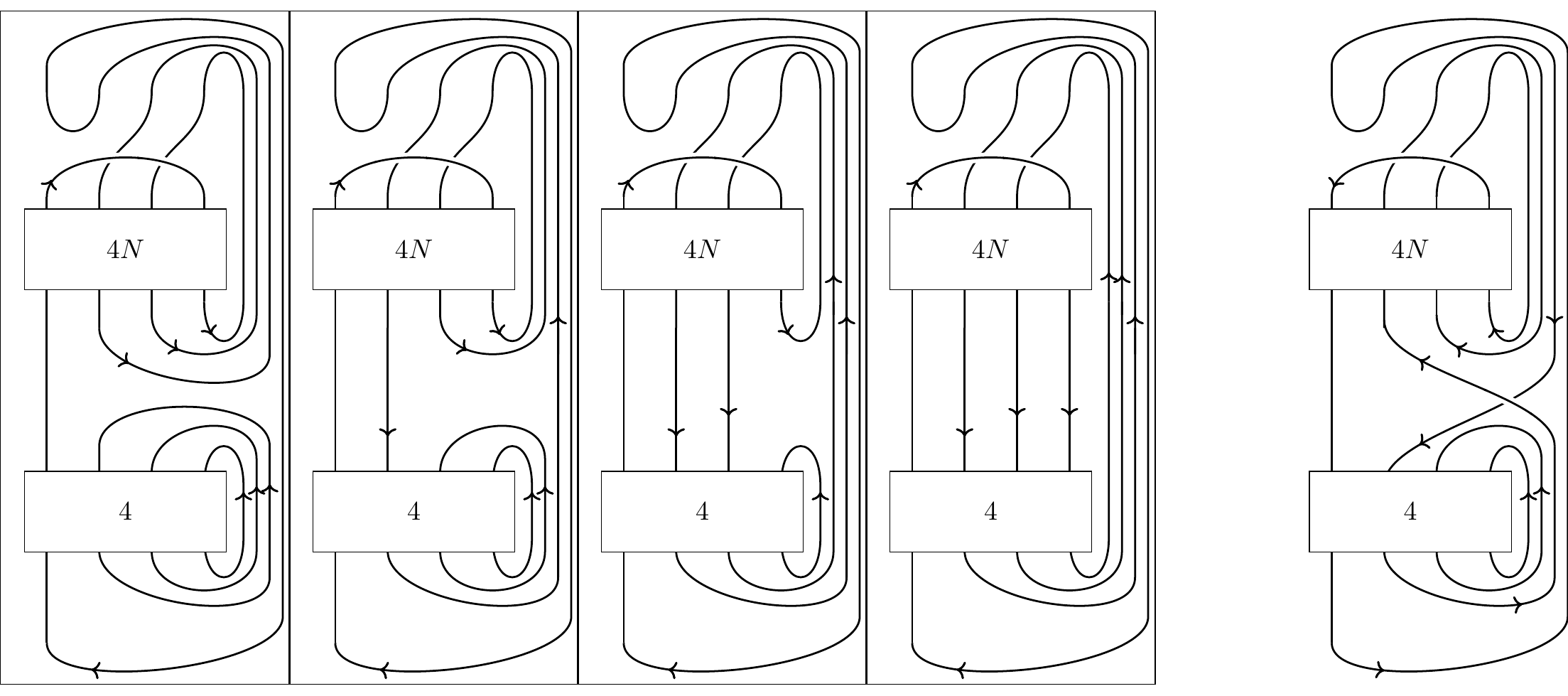}
	\captionof{figure}{On the left: the fusion movie for Lemma \ref{4surj}. On the right: the completed link associated to the first $1$-handle move.}
	\label{4movie}
\end{figure}

\begin{lemma}\label{4surj}
	The map
	\[
	\begin{aligned}
	\Sigma_{4N+1}: \widetilde{Kh}^{\ast}_{\ast}((T_{2,2N}\sharp D) \longrightarrow \widetilde{Kh}^{\ast}_{\ast}(T_{2,2(N+1)}),
	\end{aligned}
	\]
	induced by the movie in Figure \ref{4movie}, where $D$ is $T_{4,4}$ with a different orientation, are surjective.
\end{lemma}

\begin{proof}
	The proof is split into $3$ pieces, each corresponding to one of the $3$ $1$-handle moves in the movie. First notice that the map induced by the second $1$-handle realizes a connected sum, so it is surjective. The third $1$-handle is just a $2$-stranded finite fusion map
	\[
	\Sigma_2: \khr^{\ast,\ast}(T_{2,2}\sharp T_{2,2N+1}) \longrightarrow \khr^{\ast,\ast}(T_{2,2(N+1)+1}),
	\]
	hence it is surjective. If we can prove that the map induced by the first handle move is surjective, then their composition, $\Sigma_{4N+1}$, will be surjective. We will use the complete the triple procedure, which results in an exact triple 
	\[
	((U\coprod T_{2,2})\sharp T_{2,2N+1}, F \sharp T_{2,2N+1},D\sharp T_{2,2N+1}),
	\]
	where $F$ is the diagram appearing in the computation of $	\khr^{\ast}_{\ast}(T_{4,4})$ above.	Thus in a similar fashion as the previous lemma, and since there is an isomorphism
	\[
	\khr^{\ast}_{\ast}(T_{4,4}) \cong \khr^{\ast}_{\ast}((U\coprod T_{2,2})) \bigoplus \khr^{\ast}_{\ast}(F).
	\]
	We get a sequence of isomorphisms
	\[
	\begin{array}{lcl}
	\khr^{\ast,\ast}(D\sharp T_{2,2N+1}) & \cong & \khr^{\ast,\ast}(T_{2,2N+1})\otimes \khr^{\ast,\ast}(D) \\
	&\cong & \khr^{\ast,\ast}(T_{2,2N+1})\otimes \left(\khr^{\ast,\ast}((U\coprod T_{2,2})) \bigoplus \khr^{\ast,\ast}(F) \right) \\
	&\cong & \left(\khr^{\ast,\ast}(T_{2,2})\otimes\khr^{\ast,\ast}(U\coprod T_{2,2})\right) \bigoplus\left(\khr^{\ast,\ast}(T_{2,2})\otimes \khr^{\ast,\ast}(F)\right) \\
	&\cong & \khr^{\ast,\ast}(T_{2,2N+1}\sharp(U\coprod T_{2,2}))\bigoplus \khr^{\ast,\ast}(T_{2,2N+1}\sharp F).
	\end{array}
	\]
	Again we call upon Lemma \ref{homalg} to conclude that the map induced by the $1$st $1$-handle is surjective.
\end{proof}
\bibliographystyle{amsalpha}
\bibliography{bibli}
\end{document}